\documentclass[11 pt]{article}

\usepackage{amsfonts}
\usepackage{amsthm}
\usepackage{authblk}
\usepackage[numbers]{natbib}
\usepackage{mathtools}
\usepackage{braket}
\usepackage[draft]{todonotes}
\usepackage[a4paper,top=2.5cm,bottom=2.5cm,left=2.5cm,right=2.5cm,%
bindingoffset=5mm]{geometry}

\newtheorem{thm}{Theorem}[section]
\newtheorem{lemma}[thm]{Lemma}

\newtheorem{prop}[thm]{Proposition}
\newtheorem{hp}[thm]{Hypothesis}
\newtheorem{defn}{Definition}[section]

\newtheorem{remark}{Remark}[section]

\newenvironment{sistema}
{\left\lbrace\begin{array}{@{}l@{}}}
{\end{array}\right.}
\DeclarePairedDelimiter{\abs}{\lvert}{\rvert}
\DeclarePairedDelimiter{\norm}{\lVert}{\rVert}

\newcommand{\opn}[1]{\operatorname{#1}}
\newcommand{\E}{\mathbb{E}}
\newcommand{\U}{\mathcal{U}}
\newcommand{\R}{\mathbb{R}}
\newcommand{\mP}{\mathbb{P}}
\newcommand{\F}{\mathcal{F}}
\newcommand{\spazio}{\left( \Omega, \F, \mP \right)}

\def\call{{\cal L}}

\title{Stochastic maximum principle for optimal control  of a class of nonlinear SPDEs with dissipative drift
\thanks{Financial support from the grant MIUR-PRIN 2010-11
``Evolution differential problems: deterministic and stochastic approaches and their interactions''
is gratefully acknowledged. The second author have been supported by the Gruppo
Nazionale per l'Analisi Matematica, la Probabilit\`a e le loro
Applicazioni (GNAMPA) of the Istituto Nazionale di Alta Matematica
(INdAM).
}
}
\author{Marco Fuhrman \thanks{Email: marco.fuhrman@polimi.it}}
\affil{Dipartimento di Matematica, Politecnico di Milano\\ via Bonardi 9, 20133 Milano, Italia}
\author{Carlo Orrieri \thanks{Email: carlo.orrieri01@ateneopv.it}}
\affil{Dipartimento di Matematica, Universit\`a di Pavia\\ via Ferrata 1, 27100 Pavia, Italia}
\date{}
\begin{document}
\maketitle

\begin{abstract}
\noindent We prove a version of the stochastic maximum principle, in the sense of Pontryagin, for
  the finite horizon optimal control of a stochastic partial differential equation driven by an infinite dimensional additive noise. In particular we treat the case in which the non-linear term is of Nemytskii type, dissipative and with polynomial growth. The performance functional
  to be optimized is fairly general and may depend on point evaluation
  of the controlled equation. The results can be applied to a large class of non-linear parabolic equations such as reaction-diffusion equations.
\end{abstract}

\section{Introduction}

In the framework of classical optimal control theory,
the maximum principle in the sense of Pontryagin is generally understood
as a necessary condition for optimality of a control and its associated trajectory.
Under some additional assumptions, the conditions may also become sufficient
to ensure optimality.
In the context of stochastic systems, a very general formulation
of the maximum principle was obtained by  Peng in \cite{peng1990general},
in the case of a controlled finite dimensional equation driven by the Brownian motion.
This work was later extended in several directions by many authors.
Here we pay particular attention to the case of infinite dimensional controlled systems,
in particular controlled stochastic partial differential equations (SPDEs), on
a finite time horizon.

The first result in infinite dimensions is due to Bensoussan
\cite{bensoussan1983stochastic} in the case of a convex set of control actions
and a possibly infinite-dimensional Brownian noise.
Later in \cite{hu1990maximum}  the case of   diffusion coefficient unaffected
by the control parameter was treated.  Several versions of the
stochastic maximum principle for general evolution equations were proved in
 \cite{tang1993maximum}, other versions can be found in \cite{du2013maximum}, \cite{du2012stochastic} and \cite{lu2012general}, under various conditions.
The results in \citep{fuhrman2012stochastic}, \citep{fuhrman2013stochastic}
can be directly applied to a large class of concrete controlled SPDEs of parabolic type,
but only in the case of a finite-dimensional Brownian noise.
It is worth emphasizing that the general case (in which the control parameter enters the diffusion term,
the set of control actions is general --in particular non convex-- and
the noise is infinite-dimensional) is still an open problem.
\\
The aim of this work is to prove a stochastic maximum principle for a class of controlled
  semilinear SPDEs of  reaction-diffusion  type written formally as
\begin{equation}\label{SPDEintro}
\begin{sistema}
\dfrac{\partial X}{\partial t}(t,\xi) = \mathcal{A}X(t,\xi) + f(X(t,\xi),u(t)) + B \dfrac{\partial^2 w}{\partial \xi \partial t}(t,\xi)  \\
X(0,\xi) = x_0(\xi), \\
X(t,\xi) = 0 \quad\text{on} \quad \partial \mathcal{O},
\end{sistema}
\end{equation}
where $\xi \in \mathcal{O}$ a bounded  subset of $\R^d$, and $t \in [0,T]$
for some fixed $T>0$. $u(\cdot)$ is the control process, taking values in
a general space of control actions $U$,
$\mathcal{A}$ is an operator of elliptic type,
$x_0$ a given initial condition,
${\partial^2 w}/{\partial \xi \partial t}$ is a symbolic notation for
a space-time white noise, $B$ specifies the covariance operator of the noise,
which enters the equation in an additive way, and finally $f:\R\times U\to \R$
is a given function. Thus, the non-linear part of the drift is
the so-called Nemytskii operator (or superposition operator) associated to $f$.

In the context of the general theory of SPDEs, the well-posedness of
this type of equation (for a given control $u(\cdot)$) is known under
very general conditions on the function $f$, see for instance
\cite{da1996ergodicity}
for a systematic exposition.
 Particular attention is payed
to the case when $f$ is decreasing, which corresponds to a dissipativity
assumption on the non-linear part of the drift, but it is otherwise
very general, in particular it does not satisfy any Lipschitz condition.
Most of the existing results were proved
with the primary aim of studying the asymptotic behavior of an uncontrolled SPDE.
In
\cite{brzezniak2013optimal} dissipativity assumptions are required
on the drift in connection with optimal control of an SPDE, with
the aim   to prove existence of relaxed optimal controls.
The analysis is often carried out by re-writing the SDPE as an
abstract evolution equation whose trajectories lie in the space
$H:= L^2(\mathcal{O})$ or, under appropriate conditions, in the space
$E := C(\mathcal{\bar{O}})$.

On the contrary all of the previously mentioned references on the stochastic maximum principle
require the function $f$ to be at least Lipschitz continuous.
It is our purpose to prove a version of the stochastic maximum principle
by imposing to the nonlinear term only dissipativity and
polynomial growth conditions, along with some smoothness property.
Even  for finite dimensional controlled stochastic equations, this seems
to be an issue that was considered only very recently: see
 \cite{orrieri2013stochastic}. On the other hand, we limit ourselves
 to the case where the noise is additive and uncontrolled.

In order to achieve gretaer generality we formulate our results
for an abstract stochastic controlled evolution equation
with additive noise in the Banach space $E$, of the form
\begin{equation*}
\begin{sistema}
dX(t) = \left[ AX(t) + F(X(t),u(t))\right] dt + BdW(t), \\
X(0) = x_0.
\end{sistema}
\end{equation*}
We assume that $A$ is a dissipative sectorial operator in $E$, and thus
 the  the generator of an analytic contraction semigroup of linear bounded
 operators $S(t)$, $t\ge 0$. We assume that  $S$ extends to a semigroup on
 $H$ satisfying in particular the regularizing property
$S(t)(H)\subset   E$ for $t >0$, as well as other assumptions.
The non-linear term $F: E \rightarrow E$ is of Nemytskii type and it is dissipative in $E$. The noise $W$ is a cylindrical Wiener process in another separable Hilbert space $K$ and $B: K \rightarrow H$ is a  linear bounded operator. Under suitable assumptions on   $B$ we can guarantee that the stochastic convolution
$\int_0^t S(t-s)BdW(s)$,
$t\in [0,T]$,
admits an $E$-continuous version and therefore by knonw results  (see e.g. \cite{da1992stochastic}) that the state equation admits a unique mild solution in $E$ for any admissible control.
We take special care in verifying that our abstract assumptions can be effectively
checked in concrete cases like equation \eqref{SPDEintro}.

In our context the stochastic maximum principle is the statement that to any optimal pair
$(u,X)$ we can associate a pair of adjoint processes $(p,q)$ in such a way that
 the Hamiltonian function, constructed by means of $u,X$ and $p$, satisfies a maximum (or minimum)
 condition, see inequality \eqref{eq.Hamiltoniana} below.
Our main results   are two versions of the stochastic maximum principle,
with slightly different assumptions. In the first one
the adjoint processes $(p,q)$ are defined and uniquely characterized using a duality argument.
In the second one it is shown that $(p,q)$ can also be characterized as the unique mild
solution to a backward SPDE (BSPDE). While it is natural to expect that the first
adjoint process $p$ takes values
in the dual space $E'$ of $E$, it seems difficult to prove directly well-posedness of a
BSPDE as an equation in $E'$, due to the lack of an efficient stochastic calculus
in this space.  To avoid this problem we formulate the BSPDE as an equation  in
 a bigger Hilbert space where we can use standard stochastic calculus and then we prove that the solution is indeed more regular. In fact, under suitable assumptions
 we prove that the first adjoint process $p$ actually
takes values in $H'$, with a possible blow up when $t \to T$.
 In the proof we make
 frequent use of duality arguments  in order to obtain a priori estimates.

We finally mention that developing the theory in the space $E$
of continuous functions allows  to consider fairly general cost functionals. For example we are able to formulate a control problem in which we
optimize the evolution of the state evaluated at fixed points of the space, see
\eqref{costo.integrato}.

The plan of the paper is as follows. After some preliminary results
on the well-posedness of the state equation and the formulation of the
optimal control problem in Section 2, we state our main results, Theorems \ref{thm_1} and
\ref{thm_2}, in
Section 3. Section 4 is devoted to some examples, where we show that
our general results can indeed be applied  to various concrete controlled SPDEs.
  In Section 5   we analyze the spike variation technique and in Section 6
  we prove  Theorem \ref{thm_1}. Finally, in Section 7 we study the BSPDE
  for the adjoint processes, which immediately leads to the proof of Theorem \ref{thm_2}.

\section{Notations, assumptions and preliminaries}

Let $\mathcal{O} \subset \R^d$ be an open bounded subset of $\R^d$ with boundary $\partial\mathcal{O}$ of class $C^2$. We denote by $H$ the Hilbert space $L^2(\mathcal{O},\R)$ with inner product $\braket{\cdot,\cdot}_{H}$ and by $E$ the Banach space $C(\bar{\mathcal{O}},\R)$ endowed with the supremum norm $\abs{\cdot}_E$. Moreover we denote by ${}_{E'}\langle \cdot, \cdot \rangle_{E}$ ( or simply $\braket{\cdot,\cdot}_{E}$) the duality pairing in $E' \times E$, where $E'$ is the topological dual of $E$, and by ${}_{H'}\braket{\cdot,\cdot}_H$ the duality between $H$ and $H'$.
Given a probability space $\spazio$, by a cylindrical Wiener process we mean a family of linear mappings from $K$ to $L^2(\Omega)$
\[ K \ni h \mapsto	W^h_t \in L^2(\Omega) \]
such that the two following conditions hold
\begin{enumerate}
\item $\left(W_t^h \right)_{t\geq 0}$ is a real continuous Wiener process, for every $ h \in K$;
\item $\E(W_t^h \cdot W_t^k) = \braket{h,k}_{K}$, for every $ h,k \in K$.
\end{enumerate}
\begin{remark}{\em  We can also think the cylindrical Wiener process as
\[ W(t) = \sum_{k=1}^{\infty}e_kW_k(t)\]
where ${e_k}$ is a complete orthonormal system of $K$ and ${W_k(t)}$ are mutually independent real Brownian motions defined on $\spazio$. It is worth noting that the series above does not converge in $K$, but in any Hilbert space $K_1 \supset K$ such that the embedding is Hilbert-Schmidt (see \cite{da1992stochastic} for a detailed exposure).
}
\end{remark}
We use the natural filtration $(\mathcal{F}_t)_{t\geq 0}$ associated to $W$, augmented in the usual way with the family of $\mP$-null sets of $\mathcal{F}$, and we denote by $\mathcal{P}$ the progressive $\sigma$-algebra on $\Omega \times [0,T]$, for some $T>0$. If $B$ is any Banach space, for any $p \geq 1$, $\lambda\in\R$ and $T>0$ we define
\begin{itemize}
\item $L^{p}_{\mathcal{F}_T}(\Omega; B)$, the set of all
$\mathcal{F}_T$-measurable random variables $Y$ with values in $B$ such that
\[ \norm{Y}_{L^p_{\mathcal{F}_T}(\Omega; B)}
= \left( \E \abs{Y}_{B}^{p} \right)^{1/p} < \infty; \]
\item $L^2_{\mathcal{F}}(\Omega\times [0,T]; B)$, the set of all $(\mathcal{F}_t)$-progressive
processes with values in $B$ such that
\[ \norm{X}_{L_{\mathcal{F}}^2(\Omega\times[0,T];B)} = \left( \E \int_0^T \abs{X(t)}_{B}^2 dt\right)^{1/2} < \infty; \]
\item $L^p_{\mathcal{F}}(\Omega; C([0,T]; B))$ the set of all $(\mathcal{F}_t)$-adapted
continuous processes with values in $B$ such that
\[ \norm{X}_{L^p_{\mathcal{F}}(\Omega; C([0,T]; B))} = \left( \E ( \sup_{t \in [0,T]}\abs{X}_B )^p \right)^{1/p} \]
\item $L^p_{\mathcal{F}}(\Omega; L^2([0,T],\lambda; B))$, the set of all $(\mathcal{F}_t)$-progressive processes with values in $B$ such that
\[ \norm{X}_{L^p_{\mathcal{F}}(\Omega; L^2([0,T],\lambda; B))} = \left( \E \left( \int_0^T \abs{X(t)}_{B}^2(T-t)^{\lambda}dt \right)^{p/2} \right)^{1/p} < \infty. \]
\end{itemize}
\begin{remark} {\em If $B$ is a Hilbert space then also the spaces $L^2_{\mathcal{F}}([0,T]; B)$ and $L^2_{\mathcal{F}}([0,T],\lambda; B)$ are Hilbert. Moreover,
it is easy to see that
\[ [L^2_{\mathcal{F}}([0,T],\lambda; B)]' = L^2_{\mathcal{F}}([0,T],-\lambda; B) \]
where the duality is given by $\braket{x,y} = \E\int_0^T \braket{x(t),y(t)}dt$.
}
\end{remark}
If $K,H$ are
Banach spaces we denote $\call(K,H)$ the space of linear bounded operators $T$
from $K$ to $H$, endowed with the usual operator norm $\|T\|_{\call(K,H)}$.
We set $\call(K):=\call(K,K)$.
When $K,H$ are Hilbert spaces
we write $\call_2(K,H)$ for
the space of Hilbert-Schmidt operators from $K$ to $H$, that is,
linear operators $T\in \call(K,H)$	such that
\[\|T\|^2_{\call_2(K,H)}= \sum_{k \in \mathbb{N}}\abs{Te_k}^2 <\infty \]
where $(e_k)_{k \in \mathbb{N}}$ is any orthonormal basis of $K$.  We set
$\call_2(K):=\call_2(K,K)$.

The space of control actions is a general separable metric space $U$ endowed with its Borel $\sigma$-algebra $\mathcal{B}(U)$. A control process is a progressive process $(u_t)_{t \in [0,T]}$ with values in $U$. We denote with $\mathcal{U}$ the space of admissible controls.

\subsection{First assumptions on the controlled state equation}

The aim of this work is to study a controlled stochastic partial differential equation of the form
\begin{equation}\label{pr.controllo}
\begin{sistema}
dX(t) = \left[ AX(t) + F(X(t),u(t))\right] dt + BdW(t) \\
X(0) = x,
\end{sistema}
\end{equation}			
and try to give some necessary conditions for the existence of a control process which minimize a cost functional of the following type
\begin{equation}\label{costo}
J(u) = \E \int_0^T L(t,X(t),u(t))dt + \E\, G(X(T)).
\end{equation}
A control process $\bar{u}$ for which \eqref{costo} attains its minimum is called optimal, i.e.
\begin{equation*}
J(\bar{u}) = \inf_{u(\cdot) \in \U} J(u(\cdot)).
\end{equation*}
$\bar{u}$, together with its corresponding trajectory $\bar{X}$, will be called
an optimal pair $(\bar{u},\bar{X})$.

Let us give some assumptions on the equation we are considering.
\begin{hp}\label{hp_eq}
\begin{enumerate}
\item $A: D(A) \subset E \rightarrow E$
is a sectorial operator in $E$, which
generates a contraction semigroup $S(t)$, $t \geq 0$. Moreover,
 $S(\cdot)$
  extends to a strongly continuous semigroup in $H$,
  still denoted with the same symbol. We assume that $S(t)(H)\subset  E$  for $t >0$, and for some constants $C\ge 0 $ and $\lambda \in [0, 1)$ it holds that
\begin{equation}\label{ipercontratt}
      \norm{S(t)}_{\call(H, E)} \leq \dfrac{C}{t^{\lambda}}, \quad t \in (0,T].
\end{equation}

\item $B \in \mathcal{L}(K,H)$, $S(t)B\in \call_2(K,H)$ for almost
 every $t\in [0,T]$,
 \begin{equation}\label{stochconvbenposta}
    \int_0^T\|S(t)B\|_{\call_2(K,H)}^2dt<\infty
 \end{equation}
 and the stochastic convolution \[W_A(t):= \int_0^t S(t-s)BdW(s),
 \qquad t\in [0,T],\]
admits a modification with  trajectories being continuous functions with values in $E$.
\item $x\in E$.
\end{enumerate}
\end{hp}
\begin{remark} {\em
\begin{enumerate}
\item  Since $A$ is sectorial, it generates an analytic semigroup $S(\cdot)$ of bounded
linear operators in $E$, non necessarily strongly continuous, and $D(A)$ may not
be dense in $E$ (see e.g. \cite{lunardi2012analytic}). We additionally
require that the semigroup is contractive, that is it satifies
$\norm{S(t)}_{\call(E)} \leq 1$ for every $t\ge 0$, and that
it has an extension as required above.

\item
The property \eqref{ipercontratt} is a form of
 ultracontractivity, with the additional quantitative
  requirement that $\lambda <1$.
We will use this property in Proposition \ref{prop.forward.duale}.
\item Under \eqref{stochconvbenposta}, for every $t\in[0,T]$, the stochastic integral
  $W_A(t)$ is well defined as an equivalence class of  random
variables with values in $H$. We require that there exists an $E$-valued
continuous stochastic process that, considered as a process in $H\supset E$,
is a modification of $W_A$.
\item All the results that follow admit easy generalizations to the case when
the drift $F$ is  time dependent or even stochastic. Similarly, the initial
condition $x\in E$ might be replaced by a random variable.
We keep the previous setting to simplify the notation.
			
\end{enumerate}
}\end{remark}
In order to state Theorems \ref{thm_1} and \ref{thm_2} below,
we need to require that the semigroup is indeed more regular. To do it, we formulate two different hypotheses
\begin{hp}\label{hp_V1}
There exists a Hilbert  space $V$, continuously and densely embedded in $E$,
 and a constant $\beta$ such that
\begin{equation*}
S(t)( V)\subset  V, \qquad \qquad \qquad \norm{S(t)}_{\call(V)} \leq \beta,
\qquad   t \in (0,T].
\end{equation*}
\end{hp}

\begin{hp}\label{hp_V2}
There exists a Hilbert  space $V$, continuously and densely embedded in $E$,
 and a constant $\beta$ such that
\begin{equation*}
S(t) (E) \subset  V, \qquad \qquad \qquad
\norm{S(t)}_{\call(V)} \leq \beta, \qquad   t \in (0,T].
\end{equation*}
\end{hp}

The second assumption is clearly more stringent. Note however that in both
cases the bound concerns the norm in $\call(V)$.
It will be clear later on (see Proposition \ref{prop.forward.duale} and Section $7$) which are the motivations for the introduction of the new space $V$.

Regarding the cost functional, we need the following
\begin{hp}\label{hp_cost}
\begin{enumerate}
\item The functionals $L: \Omega \times [0,T] \times E \times U \rightarrow \R$ and $G: \Omega \times E \rightarrow \R$ are measurable with respect to $\mathcal{P} \otimes \mathcal{B}(E)\otimes \mathcal{B}(U)$ and $\mathcal{B}(\R)$ (respectively $\mathcal{F}_T \otimes \mathcal{B}(E)$ and $\mathcal{B}(\R)$).
\item For each $(t,u) \in [0,T] \times U$, the operators $L$ and $G$ are Fr\'echet differentiable on $E$. Moreover there exist $K \ge 0$, $k \geq 0$ such that, $\mP$-a.s.
\[\abs{L(t,x,u)}+ \abs{D_x L(t,x,u)}_E \leq K(1 + \abs{x}_E^{k}),
\qquad
t \in [0,T],\,x\in E,\, u \in U,
\]
\[\abs{D_x G(x)}_E \leq K(1 + \abs{x}_E^{k}),
\qquad
x\in E.\]
\end{enumerate}
\end{hp}

\subsection{Assumptions on the drift as a Nemytskii operator}
Here we describe the non-linear term $F$ in  the state equation.
We suppose that a function $f: \R \times U \rightarrow \R$ is given
and that
$F: E \times U \rightarrow E$ is defined by
\[ F(x,u)(\xi) = f(x(\xi),u), \qquad \xi \in \mathcal{O}\]
for any continuous map $x: \bar{\mathcal{O}} \rightarrow \R$ and $u \in U$.
Thus, $F$ is the so-called
\textit{Nemytskii operator} associated to the real function $f$, on which
we make the following assumptions.

\begin{hp}\label{hp_f}
\begin{enumerate}
\item  $f: \R \times U \rightarrow \R$ is measurable.  For every $\sigma\in\R$ the
 map $u \mapsto f(\sigma, u) $ is continuous on $U$, and for every $u\in U$
the map $\sigma \mapsto f(\sigma,u)$ is continuously
 differentiable on $ \R$, with derivative denoted by $f'(\sigma,u)$. Moreover there exist
 $C\ge 0$, $k\ge 0$ such that
\begin{equation}\label{polygrowth}
\abs{f(\sigma,u)} +\abs{f'(\sigma,u)} \leq C(1 + \abs{\sigma}^k),
\qquad \sigma \in\R,\,u \in U.
\end{equation}
\item There exists $\beta\in\R$ such that
\[  f'(\sigma,u) \leq \beta,\qquad
\sigma \in\R,\,u \in U
. \]

\end{enumerate}
\end{hp}

If the map $f$ satisfies the above assumption, the Nemytskii operator $F$
 cannot be defined as a map from $H\times U$ into $H$.
 Even if it could be defined, under more restrictive growth assumptions on $f$, it would
 fail to be differentiable in general: indeed, it is known that
a Nemytskii operator $T: H \rightarrow H$ is Fr\'echet differentiable if and only if $T$ is linear (see \cite{ambrosetti1993primer}, page 20). However, as a map from $E\times U$ to $E$,
 $F$ is well defined and has some regularity properties, as shown in the following
\begin{lemma}\label{l.multiplication}
The functional $F:E\times U\to E$ is Fr\'echet differentiable on $E$.
Moreover the Fr\'echet differential $D_xF$ acts on every $h\in E$ as a multiplication operator,
 namely
\begin{equation}\label{eq.mult_op}
\left[D_xF(x,u)\cdot h\right](\xi) = f'(x(\xi),u)\cdot h(\xi), \qquad
\xi \in \mathcal{O}.
\end{equation}
\end{lemma}
\begin{proof}
Fix $x,h\in E$, $u\in U$. For $\xi\in \mathcal{O}$, $s\neq 0$ we have
\begin{equation}
f(x(\xi) + sh(\xi),u) - f(x(\xi),u) = \int_0^1 f'(x(\xi)+\theta s h(\xi),u)sh(\xi)d\theta.
\end{equation}
Fix $\varepsilon >0$, let $\delta > 0$ be such that
$\abs{f'(x(\xi) + y,u)- f'(x(\xi),u)} < \varepsilon$, for $\abs{y} \leq \delta$. Such a $\delta$ exists since $f'(\cdot,u)$ is continuous on a compact set containing the image of $x(\cdot)$. Now choose $s$ small enough such that $\abs{\theta sh(\xi)} \leq \delta$. Taking the limit we obtain
\begin{equation*}
\begin{split}
&\sup_{\xi\in \mathcal{O}}\;\dfrac{1}{s}\abs{f(x(\xi) + sh(\xi),u) - f(x(\xi),u) - f'(x(\xi),u)\, sh(\xi)}\\
&=\sup_{\xi\in \mathcal{O}}\Big| \int_0^1[f'(x(\xi) + \theta s h(\xi),u) - f'(x(\xi),u)]d\theta \, h(\xi)\Big|\\
&\leq \varepsilon \sup_{\xi \in \mathcal{O}}\abs{h(\xi)}
\end{split}
\end{equation*}
which proves   \eqref{eq.mult_op}.
\end{proof}

To proceed further
we need to recall the following
\begin{defn}
A map $g:D(g)\subset E \rightarrow E$ is called
dissipative if for all $x,y \in D(g)$ and $\alpha \geq 0$ it holds
\[ \abs{x-y}_{E} \leq \abs{x-y-\alpha(g(x)-g(y))}_{E}. \]
Equivalently, if there exists $z^* \in \partial\abs{x-y}_E$ such that
\[ {}_{E}\braket{g(x)-g(y),z^*}_{E'}\leq 0 \qquad \forall \; x,y \in D(g).\]
Here we denote by $\partial\abs{\cdot}_E$ the sub-differential of the norm in $E$.

In the special case in which $E$ is a Hilbert space, this condition coincides with the monotonicity assumption $\braket{g(x)-g(y),x-y}_E\leq 0$, for all $x,y \in D(g)$.

\end{defn}
We refer e.g. to
\citep{cerrai2001second} for basic properties of dissipative mappings.
We also take from \citep{cerrai2001second}, page 180,
the following two results on the properties of the operator $F$
and its Yosida approximations.

\begin{lemma}\label{l.nemytskii}
Under Hypothesis \ref{hp_f}, if $F$ is the Nemytskii operator associated to $f$ then
there exists a constant $c \in \R$ such that the following are true:
\begin{itemize}
\item[(i)]  $F(\cdot,u) - cI$ is dissipative in $E$,
for every $u\in U$;
\item[(ii)] If $x \in E$ then $\sup_{u \in U}\norm{D_xF(x,u)}_{\mathcal{L}(E)} \leq c(1 + \abs{x}_E^k)$;
\item[(iii)] If $x,h \in E$ then there exists $\delta_h \in \partial\abs{h}_E$ such that $\sup_{u \in U} \braket{D_xF(x,u)h,\delta_h}_E \leq c\abs{h}_E$;
\item[(iv)] For $x,h \in E$ it holds $\sup_{u \in U}\braket{D_xF(x,u)h,h}_H \leq c\abs{h}_H^2$.
\end{itemize}
\end{lemma}

\begin{lemma}\label{l.yosida}
For any $\alpha >0$, $\sigma\in\R$ and $u \in U$,
consider the resolvent map
$J_{\alpha}(\sigma,u) = (I-\alpha f)^{-1}(\sigma,u)$ and define
$f_{\alpha}(\sigma,u) = f(J_{\alpha}(\sigma,u),u) - cJ_{\alpha}(\sigma,u)$. If we denote $F_{\alpha}(x,u)(\xi) = f_{\alpha}(x(\xi),u)$, then $F_{\alpha}$ is dissipative and Lipschitz-continuous both in $H$ and in $E$. Moreover
\begin{equation*}
[DF_{\alpha}(x,u)y](\xi) = f_{\alpha}'(x(\xi),u)\cdot y(\xi), \qquad \xi \in \mathcal{O}, u \in U.
\end{equation*}
and for any $R >0$ we have
\begin{equation*}
\lim_{\alpha \rightarrow 0} \sup_{\abs{x}_E \leq R} \abs{F_\alpha(x,u) - F(x,u)}_E = 0, \qquad u \in U.
\end{equation*}
\end{lemma}

\begin{remark}{\em
Let us notice that the operators $F_{\alpha}$ defined above coincide with the usual Yosida approximations of $F$. We also note that from \eqref{polygrowth}
it follows that
\begin{equation}\label{polygrowthdue}
\abs{f_\alpha(\sigma,u)} +\abs{f_\alpha'(\sigma,u)} \leq C(1 + \abs{\sigma}^k),
\qquad \sigma \in\R,\,u \in U.
\end{equation}
for some $C>0$ independent of $\alpha$.
}
\end{remark}
\subsection{The state equation}
Now we are in position to study the abstract form of the state equation
\eqref{pr.controllo}. While \eqref{pr.controllo} is merely a formal
writing, the precise formulation of the state equation is a so-called
 \textit{mild formulation}: an $E$-valued continuous
 adapted process $X(t)$ is a \textit{mild solution} of the SPDE above if,  $\mP$-a.s.,
\[ X(t) = S(t)x + \int_0^t S(t-s)F(X(s),u(s))ds + W_A(t),
\qquad t \in [0,T], \]
where $W_A$ is the $E$-continuous modification of $\int_0^t S(t-s)BdW(s)$, $t\in [0,T]$.
Under the previous assumptions  we can state the following theorem (compare
\cite{da1992stochastic}).
\begin{thm}
Assume that hypotheses \ref{hp_eq} and \ref{hp_f} hold true, then equation \eqref{pr.controllo} admits a unique mild solution $X$. Moreover
$X\in L^p_{\mathcal{F}}(\Omega; C([0,T]; E))$
for every  $p\geq 1$.
\end{thm}

\section{Statement of the main results}

Let $V'$ be the dual of the Hilbert space $V$ introduced in  Hypothesis \ref{hp_V1}. Let us identify $H$
with $H'$ by the Riesz isometry. Then we obtain the following continous dense inclusions
\[ V \subset E \subset H \subset E' \subset V'.\]
 Now we can state the two building blocks of our main result.
\begin{thm}\label{thm_1}
Suppose Hypotheses \ref{hp_eq}, \ref{hp_V1}, \ref{hp_cost} and \ref{hp_f} hold, let
 $  r' \in (1, 2)$ be fixed arbitrarily and let $\lambda$ be as in Hypothesis \ref{hp_eq}.
Suppose that $(u,X)$ is an optimal pair.
Then there exists a progressive process $p$
 with values in $H'$ satisfying
 \[ \E \left( \int_0^T\abs{p(t)}^2_{H'}(T-t)^{\lambda}dt\right)^{r'} <\infty\]
and  for which the following inequality holds $\mP$-a.s. for almost every $t \in [0,T]$:
\begin{equation}\label{eq.Hamiltoniana}
\mathcal{H}(t,v,X(t),p(t)) - \mathcal{H}(t,u(t),X(t),p(t)) \geq 0, \qquad \text{ for every } v \in U,
\end{equation}
where $\mathcal{H}(t,u,x,p) := L(t,x,u) + {}_{H'}\braket{p,F(x,u)}_{H}$ is the Hamiltonian
function of the system.
\end{thm}
We note that in Theorem \ref{thm_1} no
characterization of $p$ is given. It will be clear from the proof  that $p$ is   the first component of a pair $(p,q)$ which is uniquely determined by a duality argument. In
  Theorem \ref{thm_2} we give a more precise characterization of the pair $(p,q)$ as the solution of an adjoint  equation, which is backward in time.
\begin{thm}\label{thm_2}
Let $(u,X)$ be an optimal pair. If Hypotheses \ref{hp_eq}, \ref{hp_V2}, \ref{hp_cost} and \ref{hp_f} hold and $  r' \in (1, 2)$
 then there exists a pair of progressive processes $(p,q)$ with values in $H'\times L_2(K,V')$
 for which inequality \eqref{eq.Hamiltoniana} holds. Moreover
\[ \E \left( \int_0^T\abs{p(t)}^2_{H'}(T-t)^{\lambda}dt\right)^{r'} + \E\left(\int_0^T\abs{q(t)}^2_{V'}dt\right)^{r'} <\infty.\]
and  the pair $(p,q)$   is the unique mild solution of the Backward SPDE \eqref{eq.backward} below.
\end{thm}

\begin{remark}{\em
It is worth noting that Hypothesis \ref{hp_V2} is used only in the proof of uniqueness of the Backward
SPDE. For the existence part it is sufficient that the semigroup preserves $V$, i.e. Hypothesis \ref{hp_V1}.
}
\end{remark}

\section{Examples}
The aim of this section is to present some concrete examples of stochastic control problems
 for SPDEs which can be treated using our results and where the general assumptions
 stated above can be effectively cheked.
 We begin by giving some general sufficient conditions for the space-time continuity of the stochastic convolution, that we have assumed in Hypothesis \ref{hp_eq}-2. Then we focus on the case in which the sectorial operator $A$ is a realization of the Laplace operator with Dirichlet
 boundary conditions:  we give examples of covariance operators $B$ for which
 Hypothesis \ref{hp_eq}-2 can be readily verified and
 we show how  the abstract Hilbert space $V$ can be chosen.
 Last, we briefly discuss the form of the cost functionals in which we are more interested in.

\subsection{A general condition for the space-time continuity of $W_A(t)$}
In the literature, most of the sufficient condition for space-time continuity of the stochastic convolution can only be applied to one-dimensional domains $\mathcal{O} \subset \R$ and to few simple domains in higher dimensions. Here we  prove some statements concerning
more general domains in $\R^n$, $n>1$, with boundary regular enough, mainly
relying on the results from \cite{grieser2002uniform}.
Although several arguments are usual, we could not find any reference
that includes the results we are going to prove.

The first step in this direction is the study of the Ornstein-Uhlenbeck process associated with the stochastic equation
\begin{equation}
dz_t  = Az_t dt + BdW_t.
\end{equation}
The existence of mild solution $W_A(t)$ in $H$ to the above equation is easily obtained by imposing the following trace condition, for any $t \geq 0$,
\begin{equation}
\int_0^t \norm{S(s)B}_{L_2(H)} ds = \int_0^t \opn{Tr}\left[ S(s)BB^*S(s)^* \right] ds < \infty.
\end{equation}
Under this assumption, $W_A(t)$ is a mean-square continuous Gaussian process with values in H and it is given by
\begin{equation}
W_A(t) = \int_0^t S(t-s)BdW(s).
\end{equation}
For our purposes, we need $W_A$ to be more regular and we make the following
\begin{hp}\label{hp_Stoc_conv}
\begin{itemize}
\item[(i)] For any $p\geq 2$, the semigroup $S(t)$ extends uniquely to a strongly continuous semigroup in $L^p(\mathcal{O})$.
\item[(ii)] For all $\varepsilon \in (0,1)$ and $p \geq 2$, there exist $r\geq 2$ and $K$ such that the following condition holds
\[ \abs{S(t)}_{W^{\varepsilon,p}(\mathcal{O})} \leq Kt^{-\frac{\varepsilon}{r}}\abs{x}_{L^p(\mathcal{O})}, \qquad \forall \, x \in L^{p}(\mathcal{O}), \]
where with $W^{\varepsilon,p}$ we denoted the fractional Sobolev space.
\item[(iii)] There exists an orthonormal basis $e_k$ of $H$ which diagonalize both $A$ and $B$. That is, there exist two sequences of real positive numbers $\mu_k$ and real $b_k$ such that $\mu_k \nearrow +\infty$ as $k \rightarrow +\infty$ and
\[ Ae_k = -\mu_k e_k, \quad Be_k = b_k e_k, \qquad k \in \mathbb{N}; \]
\item[(iv)] Each $e_k(\cdot)$ is bounded. Let $M >0$ and $c_k$ be an increasing sequence, then it holds that
\[ \abs{e_k(\xi)} \leq Mc_k, \qquad c_k \in \mathbb{N}, \xi \in \mathcal{O}; \]
\item[(v)] 	There exists $\alpha \in (0,\frac{1}{2})$ such that
\[ \sum_{k=1}^{+\infty} b_k^2\mu_k^{2\alpha-1}c_k^2 < \infty\]
\end{itemize}
\end{hp}
Now we are in position to state the following
\begin{thm}\label{t.convolution}
Assume Hypotheses (i)-(v) hold. Then the stochastic convolution $W_A: [0,+\infty) \times \mathcal{O} \rightarrow \R$ is continuous, $\mP$-a.s. Moreover, if $p \geq 2$ we have
\[ \E\sup_{t \in [0,T]}\abs{W_A(t)}^p_E < \infty \]
\end{thm}
Before proving the theorem we recall a useful analytic lemma (cfr. \cite{da2004kolmogorov}, page 23).
\begin{lemma}\label{l.analytic}
Assume Hypotheses (i)-(ii) hold. Let $T >0$, $p \geq 2$ and $f \in L^p([0,T]\times \mathcal{O})$. If we set
\[ F(t) = \int_0^t S(t-\sigma)(t-\sigma)^{\alpha-1}f(\sigma)d\sigma, \quad t \in [0,T]; \]
then $F \in C([0,T]\times \mathcal{O})$ and there exists a constant $C_{T,p}$ such that the following estimate hold
\[ \sup_{t,\xi}\abs{F(t,\xi)}^p \leq C_{T,p}\abs{f}^p_{L^p([0,T]\times \mathcal{O})}. \]
\end{lemma}

\begin{proof}[Proof of Theorem \ref{t.convolution}]
 Using the factorization method (see e.g. \cite{da1992stochastic} for a detailed exposition), we write $W_A(t)$ in the following form
\begin{equation}
W_A(t) = \dfrac{\sin(\pi\alpha)}{\pi} \int_0^t S(t-\sigma)(t-\sigma)^{\alpha -1}Y(\sigma)d\sigma,
\end{equation}
where
\begin{equation}
Y(\sigma) = \int_0^{\sigma}S(\sigma-s)(\sigma - s)^{-\alpha}BdW(s), \quad \sigma \geq 0.
\end{equation}
If $\xi \in \mathcal{O}$, setting $Y(\sigma)(\xi) =: Y(\sigma,\xi)$ we have
\[ Y(\sigma,\xi) = \sum_{k=1}^{\infty} b_k\int_0^{\sigma} e^{-\mu_k(\sigma - s)}(\sigma - s)^{-\alpha}e_k(\xi)dW_k(s) \]
which is a gaussian random variable with zero mean and covariance given by
\begin{equation}
v_{\alpha}(\sigma,\xi) = \sum_{k=1}^{\infty} b_k^2 \int_0^{\sigma} e^{-2\mu_ks}s^{-2\alpha}\abs{e_k(\xi)}^2ds.
\end{equation}
Thanks to Hypotheses (iv)-(v) we see that
\begin{equation}
v_{\alpha}(\sigma,\xi) \leq M \sum_{k=1}^{\infty} b_k^2 \int_0^{\infty} e^{-z}\dfrac{z^{-2\alpha}}{2\mu_k^{1-2\alpha}} c_k^2 dz \leq M \sum_{k=1}^{\infty} b_k^2 \mu_k^{2\alpha -1}c_k^2 < \infty.
\end{equation}
Therefore, from the gaussianity of $Y(\cdot,\cdot)$, there exists $C_p > 0$ such that
\[ \E\abs{Y(\sigma,\xi)}^p \leq c(\E\abs{Y(\sigma,\xi)}^2)^{p/2} < C_p \]
and we have
\begin{equation*}
\E\int_0^T \int_{\mathcal{O}}\abs{Y(\sigma,\xi)}^{p}m(d\xi)dt \leq TC_pm(\mathcal{O}).
\end{equation*}
The conclusion follows from Lemma \ref{l.analytic}.
\end{proof}

\subsection{The Laplace operator}
Let us choose $A$ to be the realization of the Laplace operator $\Delta_{\xi}$ with Dirichlet boundary conditions. If $A$ is a suitable uniformly elliptic second order differential operator of negative type the same argument holds true with some modification. The state equation  reads as
\begin{equation*}
\begin{sistema}
\dfrac{\partial X}{\partial t}(t,\xi) = \Delta_{\xi} X(t,\xi) + f(X(t,\xi),u(t)) + B \dfrac{\partial^2 w}{\partial \xi \partial t}(t,\xi)  \\
X(0,\xi) = x_0(\xi),\quad \;\xi \in \mathcal{O}, \\
X(t,\xi) = 0, \qquad \quad \,t>0,\; \xi \in \partial \mathcal{O}.
\end{sistema}
\end{equation*}
In this case, the regularizing effect of the heat semigroup can be expressed by
\begin{equation}\label{eq.sobolev}
\abs{S(t)x}_{s,2} \leq ce^{-\frac{s}{2}t}(t\wedge 1)^{-\frac{s}{2}}\abs{x}_H,
\end{equation}
where $\abs{\cdot}_{s,2}$ is the norm in the fractional Sobolev space $W^{s,2}(\mathcal{O},\R)$. The Sobolev embedding theorem assures that $W^{s,2}(\mathcal{O},\R) \hookrightarrow C(\bar{\mathcal{O}},\R)$ if $s>d/2$. Then we have that
\begin{equation}
\abs{S(t)x}_{\infty} \leq ce^{-\frac{d}{4}t}(t\wedge 1)^{-\frac{d}{4}}\abs{x}_H,
\end{equation}
In order to satisfy Hypothesis \ref{hp_eq} we have to choose $d\leq 3$, therefore in this framework we are able to study stochastic equation of reaction-diffusion type only on bounded domains of dimensions $1,2$ or $3$. Thanks to the estimate \eqref{eq.sobolev}, a good choice for the Hilbert space $V$ is the fractional Sobolev space $W^{s,2}(\mathcal{O},\R)$, for which Hypotheses \ref{hp_V1} and \ref{hp_V2} are easily satisfied. Hence we end up with a triplet of the form
\[ W^{s,2}(\mathcal{O},\R) \subset C(\bar{\mathcal{O}},\R) \subset L^2(\mathcal{O},\R).\]	

Concerning the stochastic convolution $W_{\Delta}(t)$, we want to present here an explicit computation which finalizes the Hypotheses \ref{hp_Stoc_conv} presented above. Assumptions (i)-(ii) are easily verified with the choice $r =2$. The crucial hypothesis is the fourth one. Indeed the constant $M$ in (iv) can depends on the domain $\mathcal{O}$. According to a remark in \cite{grieser2002uniform}, the two basic domains one has to deal with are the square and the ball, which correspond
 to the best and worst case, respectively, as far as the growth of normalized
eigenfunctions is concerned.
Let us begin taking $\mathcal{O} = [0,\pi]^d$, the eigenfunctions of the Laplace operator are
\[e_k(\xi) = \left(\dfrac{2}{\pi}\right)^{\frac{d}{2}}\sin(k_1\xi)\cdots \sin(k_d\xi), \qquad \xi \in [0,\pi]^d, \abs{k}^2 = k_1^2 + \ldots + k_d^2. \]
hence they are uniformly bounded from above by a constant, i.e. $\abs{e_k(\xi)} \leq M$, for all $k \in \mathbb{N}$. On the other hand, for a general domain with smooth enough boundary, Grieser in \cite{grieser2002uniform} produce an estimate of form
\[ \sup_{x\in \mathcal{O}}\abs{e_k(x)} \leq M\mu_k^{\frac{d-1}{4}}. \]
It is worth noting that the bound given above is optimal in the case in which the domain is a ball, $\mathcal{O} = \lbrace x \in R^d: \abs{x}\le 1 \rbrace$. In the following we will refer to it as to the ``worst'' case.

In many applications the diffusion operator is written as power of $A$, more precisely $B = (-A)^{-\gamma}$, for some $\gamma \geq 0$. We are looking for some condition on $\gamma$ implying the continuity in space and time of the stochastic convolution. The general idea is the following: in order to have the space-time continuity it is sufficient to colour the noise a little bit more, in other words it is sufficient to increase enough the exponent $\gamma$. To do that, we recall a result proved by Agmon, see \cite{agmon1965kernels}, and we follow \cite{bonaccorsi2004integration} for the verification of the trace condition. Let $N(\mu)$ be the number of eigenvalues of $-A$ not exceeding $\mu$, then we have an asymptotic estimate of the following type
\begin{equation}\label{relation_eigenvalue}
N(\mu) = c\mu^{d/2} + o(\mu^{d/2}), \quad \text{for some } c>0.
\end{equation}
The relation above is useful in defining the measure $\nu$ on the real interval $[\mu_1,\infty)$. Here $\mu_1$ is the first (positive) eigenvalue such that
\[ \nu(D) = \sharp \lbrace k: \mu_k \in D \rbrace, \]
for every Borel subset D of $[\mu_1, \infty)$. It is obvious that $\nu([\mu_1,\mu]) = N(\mu)$. If we rewrite Hypothesis (v) in this setting we get, on one hand,
\begin{equation}
\sum_{k=1}^{\infty}\mu_k^{2\alpha -1 - 2\gamma + (d-1)/2} < \infty,
\end{equation}
for the ``worst'' case in which the domain $\mathcal{O}$ is a ball. On the other hand, If $\mathcal{O}$ is a hypercube the condition is the following
\begin{equation}
\sum_{k=1}^{\infty}\mu_k^{2\alpha -1 - 2\gamma} < \infty,
\end{equation}
Let us consider first the worst case, we get
\begin{equation}
\begin{split}
&\sum_{k=1}^{\infty}\mu_k^{2\alpha -1 - 2\gamma + (d-1)/2}\\
&= \int_{\mu_1}^{\infty}\mu^{2\alpha -1 - 2\gamma + (d-1)/2}\nu(d\mu)\\
&= - \int_{\mu_1}^{\infty}\int_{\mu}^{\infty}\dfrac{d}{ds}s^{2\alpha -1 - 2\gamma + (d-1)/2}ds\nu(d\mu)\\
&= \int_{\mu_1}^{\infty}\int_{\mu_1}^{\infty} C s^{2\alpha -2 - 2\gamma + (d-1)/2}1_{s\geq \mu}ds\nu(d\mu)\\
&= C \int_{\mu_1}^{\infty}s^{2\alpha -2 - 2\gamma + (d-1)/2} N(s)ds\\
&\leq C \int_{\mu_1}^{\infty}s^{2\alpha -2 - 2\gamma + (d-1)/2 + d/2} ds,
\end{split}
\end{equation}
Where we used  \eqref{relation_eigenvalue} in the last inequality. Obviously, Hypothesis (v) is satisfied if and only if the last integral is finite. That is, in terms of $\gamma$, if and only if
\begin{equation}
\gamma > \dfrac{2d-3}{4} + \alpha.
\end{equation}
Proceeding in the same way for $\mathcal{O} = [0,\pi]^d$ we obtain the following condition
\begin{equation}
\gamma > \dfrac{d-2}{4} + \alpha.
\end{equation}
As an example, let us consider for simplicity the two dimensional case. Here we have
\begin{itemize}
\item $\gamma > \alpha \qquad \qquad \quad$ if $\mathcal{O} = [0,\pi] \times [0,\pi]$,
\item $\gamma > 1/4 + \alpha \qquad \;$ if $\mathcal{O} = \lbrace x \in \R^2: \abs{x}\leq 1 \rbrace$,
\end{itemize}
with $\alpha \in (0,1/2)$.
It is interesting to notice that in \cite{bonaccorsi2004integration} the authors obtained the condition $\gamma > d/4 - 1/2$ (if $d =2$, this imply $\gamma >0$) for the verification of the trace condition. To obtain the space-time continuity the condition become more restricitve. In the two dimensional case, one has to choose an exponent $\gamma$ for the diffusion term that is strictly greater than $1/4$.

\subsection{Cost functional}
One of the motivations for the study of optimal control problems in the space of continuous functions is the possibility to consider a very large class of cost functionals. Here we are particularly interested in costs of Nemytskii type of the following form
\begin{equation}\label{costo.integrato}
J(u) = \E \int_0^T \int_{\mathcal{O}}l(t,X(t,\xi),u(t))\,\mu(d\xi)\,dt + \E\int_{\mathcal{O}} g(X(T,\xi)))\,\mu(d\xi),
\end{equation}
where $\mu$ is a given measure on $\mathcal{O} \subset \R^d$ and $l$, $h$ satisfy the following conditions
\begin{enumerate}
\item The functions
\[l(\omega,t,\sigma,u): \Omega \times [0,T] \times \R \times U \rightarrow \R\; \text{ and } \; g(\omega, \sigma): \Omega \times \R \rightarrow \R\]
are assumed to be measurable with respect to $\mathcal{P} \otimes \mathcal{B}(\R)\otimes \mathcal{B}(U)$ and $\mathcal{B}(\R)$ (respectively $\mathcal{F}_T \otimes \mathcal{B}(\R)$ and $\mathcal{B}(\R)$).
\item The maps $\sigma \mapsto l(\omega,t,\sigma,u)$ and $r \mapsto g(\omega, \sigma)$ are $C^1(\R)$ and there exists $K >0$, $k \geq 0$ such that, $\mP$-a.s.
\[\sup_{u \in U}\sup_{t \in [0,T]} \abs{D_\sigma l(t,\sigma,u)} \leq K(1 + \abs{\sigma}^{k}),\]
\[\abs{D_\sigma g(\sigma)} \leq K(1 + \abs{\sigma}^{k}).\]
\end{enumerate}
The two important cases we can deal with are
\begin{itemize}
\item $\mu_1 = m$ is the Lebesgue measure;
\item $\mu_2 = \sum_{i=1}^n a_i\delta_{\xi_i}$ is a linear  combination of Dirac measures
at points   $\xi_i \in \mathcal{O}$.
\end{itemize}
Observe that a cost of Nemytskii type with $\mu = \mu_2$ reduces to a sum of pointwise evaluations
 which is well defined in the Banach space $E$ of continuous functions but not
 in the space $H$. In our case, the development of a theory of stochastic optimal control problems in a Banach framework allows to control the evolution of a stochastic reaction-diffusion equation only in a finite number of points.

\section{Spike variation method}

In this section we state preliminary results needed for the proof of
Theorems \ref{thm_1} and \ref{thm_2}. We use the classical approach
based on spike variations (in time).
Throughout this section we suppose that Hypotheses \ref{hp_eq}, \ref{hp_cost}, \ref{hp_f} hold true.

Let us consider an optimal control $u(t)$ and the corresponding optimal trajectory $X(t)$.
To shorten somehow the notation, in this section we write
$u_t, X_t$ instead of  $u(t),X(t)$ and we use similar notation for other processes.
Let $E_{\varepsilon} \subset [0,T]$ be a set of measure $\varepsilon$ of the form $[t_0,t_0+\varepsilon]$, for some $t_0 \in (0,T)$, then we can introduce the spike variation process

\begin{equation}
u^{\varepsilon}_t=
\begin{cases}
u_t , & \quad \text{if } t \in [0,T]\setminus E_{\varepsilon} \\
w_t, & \quad \text{if } t \in E_{\varepsilon},
\end{cases}
\end{equation}
for some admissible control process $w$. The perturbed trajectory is denoted by $X^{\varepsilon}$. We are going to construct a new process $Y^{\varepsilon}$ which is the mild solution to the following equation, known as first variation equation:

\begin{equation}\label{eq.first.variation}
\begin{sistema}	
\dfrac{d}{dt} Y^{\varepsilon}_t = AY^{\varepsilon}_t + D_xF(X_t,u_t)Y^{\varepsilon}_t + \delta^\varepsilon F_t \\
Y^{\varepsilon}_0=0,
\end{sistema}
\end{equation}
where we used the notation $\delta^{\varepsilon}F_t := F(X_t,u_t^{\varepsilon}) - F(X_t,u_t)$. More precisely we say that $Y^{\varepsilon}$ is a mild solution to the above equation if for all $t>0$, $\mP$-a.s. we have
\begin{equation*}
Y^{\varepsilon}_t  = \int_0^t S(t-s)\big( D_xF(X_s,u_s)Y^{\varepsilon}_s + \delta^\varepsilon F_s\big) ds.
\end{equation*}
Existence and uniqueness of a mild solution of equation \eqref{eq.first.variation} with values in $E$ is well known, see e.g. \cite{cerrai2001second} for a detailed exposition. Let us recall that,
according to Lemma \ref{l.nemytskii}, for any $y \in E$
\begin{equation}\label{dissip.DF}
\braket{DF(X_t,u_t)y,\delta_{y}}_{E} \leq c\abs{y}_E, \qquad \mP-\text{a.s.}
\end{equation}
where $\delta_{y}$ is an element of the subdifferential of the norm $\partial\abs{y}_E$ and $c \in \R$.
We will use the following notation throughout the paper:
\begin{equation}\label{crescita.DF}
\abs{DF(X_t,u_t)y}_{E} \leq C_F(\omega)\abs{y}_E, \qquad \mP-\text{a.s.}
\end{equation}
where $C_F(\omega)$ denotes a suitable random variable having finite moments of any order. \\
Now we want to write the difference of the cost functional $J(u^{\varepsilon}) - J(u)$ as a function of $Y^{\varepsilon}$, up to a negligible reminder. In order to do that we need some estimate of the difference $X - X^{\varepsilon} - Y^{\varepsilon}$. More precisely we state the following
\begin{lemma}\label{l.estimates.spike}
If we define $\xi_t^{\varepsilon} := X^{\varepsilon}_t - X_t$, and $\eta_t^{\varepsilon}:= \xi^{\varepsilon}_t - Y^{\varepsilon}_t$, then the following estimates hold, for $k \ge 1$,
\begin{itemize}
\item[(i)] $\E\sup_{t \in [0,T]}\abs{\xi_t^{\varepsilon}}^{2k}_E = O(\varepsilon^{2k})$,
\item[(ii)] $\E\sup_{t \in [0,T]}\abs{Y^{\varepsilon}_t}^{2k}_E = O(\varepsilon^{2k})$,
\item[(iii)] $\E\sup_{t \in [0,T]}\abs{\eta^{\varepsilon}_t}^{2k}_E = o(\varepsilon^{2k}).$
\end{itemize}
\end{lemma}

\begin{proof}
$(i)$. Let us denote for brevity $z_t := \int_0^t e^{(t-s)A}QdW(s)$ and consider $v_t := X_t - z_t$, $v^{\varepsilon}_t:= X^\varepsilon_t - z_t$. Then 
\begin{equation*}
v_t = \int_0^t S(t-s)F(v_s + z_s, u_s)ds, \qquad v^{\varepsilon}_t = \int_0^t S(t-s) F(v_s^{\varepsilon} + z_s, u^{\varepsilon}_s)ds.
\end{equation*}
Let $v^{n,\varepsilon}, v^{n}$ solutions to the corresponding equations with $S(\cdot)$ replaced by $S_n(\cdot)$. Therefore $v^{\varepsilon,n}$ and $v^{n}$ are solutions of the equations
\begin{equation}
\frac{d}{dt}v_t^n = A_n v_t^n + F(v^n_s + z_s, u_s), \qquad \frac{d}{dt}v_t^{n,\varepsilon} = A_n v_t^{n,\varepsilon} + F(v^{n,\varepsilon} + z_s, u^\varepsilon_s)
\end{equation}
where $A_n = nA(n-A)^{-1}$ are the Yosida approximations of the operator $A$.
Then, if  $\delta_t \in \partial\abs{v^{n,\varepsilon}_t - v_t^n}_E$, we have
\begin{equation*}
\begin{split}
\dfrac{d^-}{dt}\abs{v^{n,\varepsilon}_t - v_t^n}_E &\leq \braket{A_n\left( v^{n,\varepsilon}_t - v^{n}_t \right) + F(v^{n,\varepsilon}_t + z_t,u^{\varepsilon}_t) - F(v^n_t + z_t,u_t), \delta_t}_E \\
&\le \braket{F(v^{n,\varepsilon}_t + z_t,u^{\varepsilon}_t) - F(v^n_t + z_t,u^{\varepsilon}_t), \delta_t}_E + \braket{F(v^n_t + z_t,u^\varepsilon) - F(v^n_t + z_t,u_t), \delta_t}_E \\
&\leq c\abs{v^{n,\varepsilon}_t - v^n_t}_E + \abs{F(v^n_t + z_t,u^\varepsilon) - F(v^n_t + z_t,u_t)}_E, \\
\end{split}
\end{equation*}
from the contraction property of the semigroup and the dissipativity of the operator $F$. Thanks to the Gronwall lemma we obtain
\begin{equation*}
\abs{v^{n,\varepsilon}_t - v^n_t}_E \leq c\int_0^t \abs{F(v^n_t + z_t,u^\varepsilon) - F(v^n_t + z_t,u_t)}_E ds.
\end{equation*}
Letting $n\to \infty$
\begin{equation}\label{eq.xi}
\abs{v^{\varepsilon}_t - v_t}_E \leq c\int_0^t \abs{\delta^{\varepsilon}F_s}_E  ds.
\end{equation}
Now recall that $\xi_t^\varepsilon := X^\varepsilon_t - X_t = v^\varepsilon_t - v_t$, using the polynomial growth of $F$ we end up with 
\[ \abs{\xi_t^\varepsilon}_E \leq c \int_{E_{\varepsilon}} \abs{\delta^{\varepsilon}F_s}_E  ds \leq C_F(\omega)\varepsilon, \]
where $C_F(\omega)$ has finite moments of any order. Hence $\E\sup_{t}\abs{\xi^{\varepsilon}_t}^{2k}_E \leq C_1\varepsilon^{2k}$, $k \ge 1$. \\
$(ii)$.  Let $Y^{n,\varepsilon}$ be the solution to the following equation
\begin{equation}
\frac{d}{dt} Y_t^{n,\varepsilon} = A_nY^{n,\varepsilon}_t + D_xF(X_t,u_t)Y^{n,\varepsilon}_t + \delta^\varepsilon F_t. 
\end{equation}
Using the dissipativity assumptions we get
\begin{equation*}
\begin{split}
\dfrac{d^-}{dt}\abs{Y^{n,\varepsilon}_t}_E &\leq \braket{A_nY^{n,\varepsilon}_t + D_xF(X^{}_t,u_t)Y^{n,\varepsilon}_t + \delta^{\varepsilon}F_t, \delta_{Y^{n,\varepsilon}_t}}_E \\
&\leq c\abs{Y^{n,\varepsilon}_t}_E + \abs{\delta^{\varepsilon}F_t}_E, \\
\end{split}
\end{equation*}
then using the same strategy as before and passing to the limit with $n \rightarrow \infty$ we get the required result.\\
$(iii)$. Now let us define $\eta^{n,\varepsilon}_t := v^{n,\varepsilon}_t - v^n_t - Y^{n,\varepsilon}_t$. If $\delta_t \in \partial\abs{\eta^{n,\varepsilon}_t}$, we have
\begin{equation*}
\begin{split}
\dfrac{d^-}{dt}\abs{\eta^{n,\varepsilon}_t}_E &\leq \langle A_n\eta^{n,\varepsilon}_t + F(v^{n,\varepsilon}_t + z_t, u^{\varepsilon}_t) - F(v^n_t + z_t,u^{\varepsilon}_t) \\
&\quad - D_xF(v_t + z_t,u_t)(v^{n,\varepsilon}_t - v^n_t) + D_xF(v_t + z_t,u_t)\eta^{n,\varepsilon}_t, \delta_t\rangle_E \\
&\leq \langle \int_0^1 D_xF\left( v^n_t + z_t  + \theta \left(v^{n,\varepsilon}_t - v^n_t\right),u^{\varepsilon}_t\right)\left(v^{n,\varepsilon}_t - v^n_t\right) d\theta \\
&\quad -D_xF(v_t + z_t,u_t)(v^{n,\varepsilon}_t - v^n_t),\delta_t \rangle_E  + c\abs{\eta^{n,\varepsilon}_t}_E\\
&= \langle \int_0^1 \left[ D_xF\left( v^n_t + z_t  + \theta \left(v^{n,\varepsilon}_t - v^n_t\right),u^{\varepsilon}_t\right) - D_xF\left( v^n_t + z_t,u^{\varepsilon}_t\right)\right]\left(v^{n,\varepsilon}_t - v^n_t\right) d\theta,\delta_t \rangle_E \\
&\quad +\langle \left[ D_xF(v^n_t + z_t,u^\varepsilon_t) - D_xF(v_t + z_t,u^\varepsilon_t)\right](v^{n,\varepsilon}_t - v^n_t),\delta_t \rangle_E\\
&\quad + \delta^{\varepsilon}D_xF_t\cdot (v^{n,\varepsilon}_t - v^n_t),\delta_t \rangle_E  + c\abs{\eta^{n,\varepsilon}_t}_E.\\
\end{split}
\end{equation*}
Thanks to the Gronwall lemma we obtain
\begin{equation*}
\begin{split}	
\abs{\eta^{n,\varepsilon}_t}_E &\leq C\int_0^t \int_0^1 \Big| D_xF\left( v^n_s + z_s  + \theta \left(v^{n,\varepsilon}_t - v^n_t\right),u^{\varepsilon}_s\right) - D_xF\left( v^n_s + z_s,u^{\varepsilon}_t\right)\Big|_{\mathcal{L}(E)}\big|v^{n,\varepsilon}_s - v^n_s\big|_E d\theta ds\\
&+ C\int_0^t \Big| D_xF(v^n_s+ z_s,u^\varepsilon_s) - D_xF(v_s + z_s,u^\varepsilon_s) \Big|_{\mathcal{L}(E)}\big|v^{n,\varepsilon}_s - v^n_s\big|_E ds \\
&+ C\int_{E_{\varepsilon}}\Big|\delta^{\varepsilon}D_xF_s\Big|_{\mathcal{L}(E)}\big| v^{n,\varepsilon}_s - v^n_s \big|_E ds\\
\end{split}
\end{equation*}
Letting $n \to 0$
\begin{equation*}
\begin{split}	
\abs{\eta^{\varepsilon}_t}_E &\leq C\int_0^t \int_0^1 \Big| D_xF\left( v_s + z_s  + \theta \xi^\varepsilon_t,u^{\varepsilon}_s\right) - D_xF\left( v_s + z_s,u^{\varepsilon}_t\right)\Big|_{\mathcal{L}(E)}\abs{ \xi^\varepsilon_t}_E d\theta ds\\
&+ C\int_{E_{\varepsilon}}\Big|\delta^{\varepsilon}D_xF_s\Big|_{\mathcal{L}(E)}\abs{\xi^\varepsilon_t}_E ds,\\
\end{split}
\end{equation*}
thanks to the continuity of the map $DF(\cdot): E \rightarrow \mathcal{L}(E)$. Now according to the estimate obtained in $(i)$ and to the H\"older inequality we obtain
\begin{equation*}
\begin{split}	
\E \sup_{t \in [0,T]} \abs{\eta^{\varepsilon}_t}^{2k}_E &\leq C \varepsilon^{2k} \cdot \left( \E \int_0^T \int_0^1 \Big|D_xF(v_s + z_s + \theta\xi^{\varepsilon}_s,u^{\varepsilon}_s) - D_xF(v_s + z_s,u^{\varepsilon}_s)\Big|^{4k}_{\mathcal{L}(E)}d\theta ds \right)^{1/2} \\
&+ C\varepsilon^{2k} \left( \E \int_{E_{\varepsilon}}\Big|\delta^{\varepsilon}D_xF_s\Big|_{\mathcal{L}(E)}^{4k}ds \right)^{1/2} .
\end{split}
\end{equation*}
The continuity of the map $DF(\cdot): E \rightarrow \mathcal{L}(E)$ implies again that $\abs{D_xF(v_t + z_t + \theta\xi_t,u^{\varepsilon}_t) - D_xF(v_t + z_t,u^{\varepsilon}_t)}_{\mathcal{L}(E)}$ tends to zero if $\varepsilon \rightarrow 0$. Thanks to the polynomial growth of $D_xF$ we get the result.
\end{proof}

Now we can deal with the expansion of the cost.

\begin{prop}\label{prop.cost.expansion}
We have the following
\begin{equation*}
J(u^{\varepsilon}) - J(u) = \E\int_0^T \left[ \delta^{\varepsilon}L_t + D_xL(t,X_t,u_t)Y^{\varepsilon}_t\right]dt + \E \left( D_xG(X_T)Y^{\varepsilon}_T \right) + o(\varepsilon),
\end{equation*}
where $\delta^{\varepsilon}L_t = L(t,X_t,u^{\varepsilon}_t) - L(t,X_t,u_t).$
\end{prop}
\begin{proof}
The difference between the two cost functionals reads as
\begin{equation*}
J(u^{\varepsilon}) - J(u) = \E\int_0^T \left[L(t, X_t^{\varepsilon},u^{\varepsilon}_t) - L(t,X_t,u_t) \right] dt + \E\left[ G(X_T^{\varepsilon}) - G(X_T) \right]
\end{equation*}
Let us begin rewriting the running cost part
\begin{equation*}
\begin{split}
&\E\int_0^T \left[L(t,X_t^{\varepsilon},u^{\varepsilon}_t) - L(t,X_t,u^{\varepsilon}_t) \right]dt + \E\int_0^T\delta^{\varepsilon}L_t  dt\\
&= \E\int_0^T \int_0^1 D_xL(t,X_t + \theta\xi^{\varepsilon}_t,u_t^{\varepsilon})\xi^{\varepsilon}_t d\theta dt  + \E\int_0^T\delta^{\varepsilon}L_tdt\\
&= \E\int_0^T\int_0^1 \left[D_xL(t,X_t + \theta\xi^{\varepsilon}_t,u_t^{\varepsilon}) - D_xL(t,X_t,u_t^{\varepsilon})\right]d\theta\xi^{\varepsilon}_tdt \\
&+ \E\int_0^T D_xL(t,X_t,u^{\varepsilon}_t)\xi^{\varepsilon}_t dt
+ \E\int_0^T \delta^{\varepsilon}L_t  dt\\
\end{split}	
\end{equation*}
\begin{equation*}
\begin{split}
&\leq C \left(\E\sup_{t \in [0,T]}\abs{\xi_t^{\varepsilon}}^2_E\right)^{1/2}\left(\E\left( \int_0^T \Big| \int_0^1 \left[D_xL(t,X_t + \theta\xi^{\varepsilon}_t,u_t^{\varepsilon}) - D_xL(t,X_t,u_t^{\varepsilon})\right]d\theta \Big| dt \right)^2 \right)^{1/2}\\
&+ \E\int_0^T D_xL(t,X_t,u^{\varepsilon}_t)Y^{\varepsilon}_t dt + \E\int_0^T  D_xL(t,X_t,u^{\varepsilon}_t)\eta^{\varepsilon}_t dt + \E\int_0^T \delta^{\varepsilon}L_t  dt\\
&\leq C \varepsilon \cdot \left(\E\left( \int_0^T \Big| \int_0^1 \left[D_xL(t,X_t + \theta\xi^{\varepsilon}_t,u_t^{\varepsilon}) - D_xL(t,X_t,u_t^{\varepsilon})\right]d\theta \Big| dt \right)^2 \right)^{1/2}\\
&+ (\E(\sup_{t \in [0,T]}\abs{\eta_t}_E^{\varepsilon})^2)^{1/2} \left(\E \left(\int_0^T D_xL(t,X_t,u^{\varepsilon}_t)dt\right)^2\right)^{1/2}\\
&+\E\int_0^T D_xL(t,X_t,u^{\varepsilon}_t) Y^{\varepsilon}_t dt + \E\int_0^T \delta^{\varepsilon}L_t  dt \\
&=\E\int_0^T D_xL(t,X_t,u^{\varepsilon}_t) Y^{\varepsilon}_t dt + \E\int_0^T \delta^{\varepsilon}L_t  dt  + o(\varepsilon),  \\
\end{split}
\end{equation*}
where we used the estimates given in Lemma \ref{l.estimates.spike}, the continuity of the map $DL(\cdot): E \to \mathcal{L}(E)$ and the polynomila growth. Regarding the second term we have
\begin{equation*}
\begin{split}
&\E\left[ G(X_T^{\varepsilon}) - G(X_T) \right] = \E\int_0^1 D_xG(X_T + \theta\xi^{\varepsilon}_T)d\theta\cdot\xi^{\varepsilon}_T \\
&= \E\int_0^1 \left[  D_xG(X_T + \theta\xi^{\varepsilon}_T) - D_xG(X_t)\right]d\theta\cdot\xi^{\varepsilon}_T + \E \left( D_xG(X_T)\xi^{\varepsilon}_T\right)\\
&\leq C\varepsilon \cdot \left( \E \left( \int_0^1 \left[  D_xG(X_T + \theta\xi^{\varepsilon}_T) - D_xG(X_t)\right]d\theta \right)^2 \right)^{1/2}\\
&+ \E \left( D_xG(X_T)Y^{\varepsilon}_T \right)+ \E \sup_{t \in [0,T]}\abs{\eta^{\varepsilon}_T}_E D_xG(X_t)\\
&= \E \left(D_xG(X_T)Y^{\varepsilon}_T\right) + o(\varepsilon),
\end{split}
\end{equation*}
where  we used again the continuity of the map $DG(\cdot): E \to \mathcal{L}(E)$ along with Lemma \ref{l.estimates.spike}. By adding the two terms we  conclude the proof.
\end{proof}

\begin{remark}
In the particular case in which the cost is given by \eqref{costo.integrato}
the expansion of the cost reads as
\begin{equation*}
\begin{split}
J(u^{\varepsilon}) - J(u) &= \E\int_0^T \int_{\mathcal{O}}\left[ l(t,X_t(\xi),u^\varepsilon_t) - l(t,X_t(\xi),u_t) + D_xl(t,X_t(\xi),u_t)Y^{\varepsilon}_t(\xi)\right] \mu(d\xi)dt \\
&+ \E\int_{\mathcal{O}}Dg(X_T(\xi))Y^{\varepsilon}_T(\xi) \mu(d\xi) + o(\varepsilon).
\end{split}
\end{equation*}
\end{remark}

\section{Proof of Theorem \ref{thm_1}}

The central idea in this proof is the construction of a pair of processes $(p,q)$
 by a duality argument.  We introduce the following auxiliary equation:
\begin{equation}\label{eq.y}
\begin{sistema}
dy(t) = [Ay(t) + D_xF(X(t),u(t))y(t) + \gamma(t)]dt + \eta(t)dW(t)\\
y(0) = 0,
\end{sistema}
\end{equation}
which is a generalized version of the first variation equation introduced before.
We want to study the linear map
\[ \tau: (\gamma(\cdot),\eta(\cdot))\longmapsto (y(\cdot),y(T)), \]
which assigns to the forcing terms $\gamma(\cdot), \eta(\cdot)$ the solution $y(\cdot)$ and
 its terminal value $y(T)$. For an appropriate choice of functional
 spaces, $\tau$ turns out to be well defined and continuous, and
 so is its adjoint \[ \tau^*: (f(\cdot), \zeta) \longmapsto (p(\cdot),q(\cdot)). \]
 The pair we are looking for is defined as
  $(p,q)=\tau^*(f,\zeta)$
when we choose $\zeta = D_xH(X(T))^*$ and $f(t)= D_xL(t,X(t),u(t))^*$.
The following Proposition will allow to identify the appropriate norms
for this duality argument.

\begin{prop}\label{prop.forward.duale}
Let Hypotheses  \ref{hp_eq}, \ref{hp_V1} and \ref{hp_f} be in force.
Suppose we are given $r > 2$ and  two processes  $\gamma  , \eta  $   such that $\gamma \in L_{\mathcal{F}}^r(\Omega; L^2([0,T],-\lambda;H))$ and $\eta \in L_{\mathcal{F}}^r(\Omega; L^2([0,T];\mathcal{L}_2(K,V)))$.
\begin{enumerate}
  \item The equation \eqref{eq.y}
has a unique mild solution, i.e. a progressive $E$-valued process
 $y$ such that,  $\mP$-a.s., $t\mapsto y(t)$ belongs to
  $ L^2 ([0,T];E)$ and
\begin{equation}\label{eq.y.mild}
y(t) = \int_0^t S(t-s)\big( D_xF(X(s),u(s))y(s) + \gamma(s)\big)ds + \int_0^t S(t-s)\eta(s)dW(s),
\end{equation}
  for almost every  $t \in [0,T]$.
  \item The following estimate hold:
\begin{equation}\label{estimate.forward}
\E\int_0^T\abs{y(t)}_E^2 dt \leq C\bigg[ \E \left( \int_0^T \abs{\gamma(s)}^2_H(T-s)^{-\lambda}ds \right)^{r/2} + \E \left( \int_0^T \norm{\eta(s)}^2_{L_2(K,V)}ds \right)^{r/2} \bigg]^{\frac{2}{r} }
\end{equation}
If the forcing term $\gamma(\cdot)$ belongs to $L^r(\Omega; C([0,T];H))$
then the solution is bounded in $L^2(\Omega; E)$, more precisely the following estimate holds
for every $t\in [0,T]$:
\begin{equation}\label{estimate.forward.T}
\E\abs{y(t)}_E^2 \leq C\bigg[ \E \left( \int_0^t \abs{\gamma(s)}^2_H(t-s)^{-\lambda}ds \right)^{r/2} + \E \left( \int_0^t \norm{\eta(s)}^2_{L_2(K,V)}ds\right)^{r/2}\bigg]^{\frac{2}{r}}
\end{equation}

\end{enumerate}

\end{prop}

\begin{proof}
\textbf{Existence and uniqueness.}
Let us denote
\begin{equation*}
T(t):= D_xF(X(t),u(t)), \qquad \Gamma(t) := \int_0^tS(t-s)\gamma(s)ds, \qquad W_{\eta}(t) := \int_0^tS(t-s)\eta(s)dW(s).
\end{equation*}
We note for further use that, for $\mP$-almost every $\omega \in \Omega$,
$\sup_{t\in [0,T]}\|T(t)\|_{\mathcal{L}(E)}$ is bounded
by a constant $ C_F(\omega)$.
Then equation  \eqref{eq.y.mild} can be written:
$\mP$-a.s.,
\begin{equation}
y(t) = \int_0^t S(t-s)T(s)y(s)ds + \Gamma(t) + W_{\eta}(t) \qquad
  \text{ for almost every }  t \in [0,T].
\end{equation}
Now we want to prove that $\Gamma$ and $W_{\eta}$ belong to $L_{\mathcal{F}}^2([0,T];E)$.
\begin{equation}\label{eq.preliminary.gamma}
\begin{split}
\E \int_0^T \Big|\int_0^t S(t-s)\gamma(s)ds\Big|_E^2 dt &\leq \E \int_0^T \left( \int_0^t \abs{S(t-s)\gamma(s)}_E ds \right)^2 dt \\
&\leq C \E \int_0^T \left( \int_0^t \abs{\gamma(s)}_H(t-s)^{-\lambda} ds \right)^2 dt\\
&\leq C \E \int_0^T \left(\int_0^t \abs{\gamma(s)}^2_H (t-s)^{-\lambda} ds \cdot \int_0^t (t-s)^{-\lambda}ds \right)dt \\
&\leq C \E \int_0^T \left(\int_0^t \abs{\gamma(s)}^2_H(t-s)^{-\lambda} ds \right) dt \\
&= C \E \int_0^T \abs{\gamma(s)}_H^2 \left( \int _s^T (t-s)^{-\lambda}dt \right) ds \\
&= C \E \int_0^T \abs{\gamma(s)}_H^2 (T-s)^{1-\lambda}ds < \infty
\end{split}
\end{equation}
thanks to the fact that $\lambda <1$ and
$\gamma \in L_{\mathcal{F}}^r(\Omega;L^2([0,T],-\lambda;H))
\subset L_{\mathcal{F}}^2([0,T] ;H))$.
Regarding the stochastic convolution we have
\begin{equation}\label{eq.preliminary.eta}
\begin{split}
\E \int_0^T \Big|\int_0^t S(t-s)\eta(s)dW_s\Big|_E^2 dt &\leq C\E \int_0^T \Big|\int_0^t S(t-s)\eta(s)dW_s\Big|_V^2 dt \\
&\leq C \E\int_0^T\norm{\eta(s)}^2_{L_2(K,V)}ds < \infty
\end{split}
\end{equation}
where we have used the Ito isometry for the stochastic integral
in the Hilbert space $V$ with respect to cylindrical noise, as well as
Hypothesis \ref{hp_V1}.
If we define
\[ v(t)= y(t) - \Gamma(t) - W_{\eta}(t),\]
then equation \eqref{eq.y.mild} is equivalent  to the following: $\mP$-a.s.,
\begin{equation}\label{rand_eq_variaz_prima}
v(t) = \int_0^t S(t-s)T(s)v(s)ds + \int_0^t S(t-s)T(s)(\Gamma(s) + W_{\eta}(s))ds
\end{equation}
  for almost every  $t \in [0,T]$. Recalling that
  $\sup_{t\in [0,T]}\|T(t)\|_{\mathcal{L}(E)}\le  C_F(\omega)$ we see
  that $\mP$-a.s. $s\mapsto T(s)(\Gamma(s) + W_{\eta}(s))$
  belongs to $L^2([0,T];E)$ and then
  it can be proved by an easy contraction argument that there exists
  a unique pathwise solution $v(\omega,\cdot) \in L^2([0,T];E)$
  (in fact, $v(\omega,\cdot) \in C([0,T];E)$)
   and that, $\mP$-a.s., equality \eqref{rand_eq_variaz_prima} holds for every $t\in [0,T]$.

   By the same arguments, uniqueness of the pathwise solution $v$
   implies uniqueness for the original equation \eqref{eq.y.mild}.

\textbf{Estimates.}
Let us now define the Yosida approximations $A_n = nA(n-A)^{-1}$ of $A$ and consider the approximating equations
\begin{equation}
v_n'(t) = A_nv_n(t) + T(t)v_n(t) + T(t)(\Gamma(t) + W_\eta (t)), \qquad v(0) = 0.
\end{equation}
Then, for any $t \in [0,T]$, we have $\mP$-a.s.
\begin{equation}
\dfrac{d}{dt}^-\abs{v_n(t)}_E \leq \braket{A_nv_n(t),\delta_n(t)}_E + \braket{T(t)(v_n(t) + \Gamma(t) + W_\eta (t)),\delta_n(t)}_E
\end{equation}
where $\delta_n(t) \in \partial\abs{v_n(t)}_E$.
Using the dissipativity properties proved in Lemma   \ref{l.nemytskii}
we have $\braket{T(t)v_n(t)  ,\delta_n(t)}_E \le    c{v_n(t)}_E$ and
from the contraction property of the semigroup it follows that
$\braket{A_nv_n(t),\delta_n(t)}_E \le 0$. So
we obtain
\begin{equation}
\dfrac{d}{dt}^-\abs{v_n(t)}_E \leq c{v_n(t)}_E + C_F(\omega)\abs{(\Gamma(t) + W_\eta (t))}_E, \qquad \mP\text{-a.s.}
\end{equation}
where the constant $C_F(\omega)$ is finite for a.e. $\omega \in \Omega$. Using Gronwall's lemma
\begin{equation}
\abs{v_n(t)}_E \leq C(\omega)\int_0^t \abs{(\Gamma(s) + W_\eta (s))}_E ds, \qquad \mP\text{-a.s.}
\end{equation}
Now, it is easy to prove that  $v_n(t)\to v(t)$ in $E$ so we can take the limit as $n$ goes to infinity and we obtain
\begin{equation}
\abs{v(t)}_E \leq C(\omega)\int_0^t \abs{(\Gamma(s) + W_\eta (s))}_E ds.
\end{equation}
Recalling the definition of $v(t)$ we get
\begin{equation}
\begin{split}
\abs{y(t)}_E &\leq C(\omega)\int_0^t \abs{(\Gamma(s) + W_\eta (s))}_E ds + \abs{\Gamma(t)}_E + \abs{W_\eta (t)}_E \\
&\leq C(\omega)\left( \abs{\Gamma(t)}_E + \abs{W_\eta (t)}_E \right), \qquad \mP\text{-a.s.}
\end{split}
\end{equation}
To prove estimate \eqref{estimate.forward} we have to consider
\begin{equation*}
\begin{split}
\E \int_0^T \abs{y(t)}_E^2 dt &\leq \E \left[ C(\omega)^2 \int_0^T \left( \abs{\Gamma(t)}_E + \abs{W_\eta (t)}_E \right)^2 dt\right] \\
&\leq \left[\E (C(\omega)^{2r'}) \right]^{1/r'} \left[ \E \left( \int_0^T \left( \abs{\Gamma(t)}_E^2 + \abs{W_\eta (t)}^2_E \right) dt \right)^{r}\right]^{1/r} \\
&\leq K  \left[ \E \left( \int_0^T \abs{\Gamma(t)}_E^2 + \abs{W_\eta (t)}^2_E dt \right)^r \right]^{1/r}
\end{split}
\end{equation*}
where we chose $r \geq 2$ and we used the finiteness of the moments of $C(\omega)$. Using the same computations as in \eqref{eq.preliminary.gamma} and \eqref{eq.preliminary.eta} we end up with
\begin{equation*}
\left( \E \int_0^T\abs{y(t)}_E^2 dt \right)^{\frac{1}{2}} \leq C\Big[ \E \left( \int_0^T \abs{\gamma(s)}^2_H(T-s)^{-\lambda}ds \right)^{r/2} + \E \left( \int_0^T \norm{\eta(s)}^2_{L_2(K,V)}ds \right)^{r/2} \Big]^{\frac{1}{r}}
\end{equation*}
which is exactly what we looked for.

Assuming now  that $\gamma \in L^r(\Omega; C([0,T];H))$ then  we have
\begin{equation*}
\begin{split}\abs{\Gamma(t)}_E^2=
\Big|\int_0^t S(t-s)\gamma(s)ds\Big|_E^2 &\leq \left( \int_0^t \abs{S(t-s)\gamma(s)}_E ds \right)^2\\
&\leq C\left( \int_0^t \abs{\gamma(s)}_H(t-s)^{-\lambda} ds \right)^2\\
&\leq C\int_0^t \abs{\gamma(s)}^2_H(t-s)^{-\lambda} ds \cdot \int_0^t (t-s)^{-\lambda}ds\\
&\leq C \int_0^t \abs{\gamma(s)}^2_H(t-s)^{-\lambda} ds
\end{split}
\end{equation*}
thanks to the fact that $\lambda < 1$.
Regarding the stochastic convolution we have
\begin{equation*}
\begin{split}
\E\abs{W_\eta(t)}_E^2=\E \Big|\int_0^t S(t-s)\eta(s)dW(s)\Big|^2_E &\leq C\cdot\E \Big|\int_0^t S(t-s)\eta(s)dW(s)\Big|^2_V \\
&\leq C\cdot \E\int_0^t\norm{\eta(s)}^2_{L_2(K,V)}ds.
\end{split}
\end{equation*}
Then we have
\begin{equation*}
\begin{split}
\E\abs{y(t)}_E^2 &\leq \E \left( C(\omega)^2 \left( \abs{\Gamma(t)}_E^2 + \abs{\eta(t)}_E^2\right) \right) \\
&\leq \left[\E (C(\omega)^{2r'}) \right]^{1/r'} \left[ \E \left( \abs{\Gamma(t)}_E^2 + \abs{W_\eta (t)}^2_E \right)^{r}\right]^{1/r} \\
&\leq K  \left[ \E \left( \abs{\Gamma(t)}_E^2 + \abs{W_\eta (t)}^2_E \right)^r \right]^{1/r}.
\end{split}
\end{equation*}
The inequality
\eqref{estimate.forward.T}
now follows immediately.
\end{proof}
By Proposition \ref{prop.forward.duale},   $\tau$ is a bounded linear operator   from the space $L_{\mathcal{F}}^r(\Omega; L^2([0,T],-\lambda; H)) \times L_{\mathcal{F}}^r(\Omega; L^2([0,T]; L_2(K,V)))$ to $L^2_{\mathcal{F}}(\Omega\times [0,T]; E)\times L^2_{\mathcal{F}_T}(\Omega; E)$.
Its  dual operator $\tau^*$  is also a bounded linear map
\begin{equation*}
\tau^*: L^2_{\mathcal{F}}(\Omega\times [0,T]; E)' \times L^2_{\mathcal{F}_T}(\Omega; E)' \longrightarrow L_{\mathcal{F}}^{r'}(\Omega; L^2([0,T],-\lambda; H')) \times L_{\mathcal{F}}^{r'}(\Omega; L^2([0,T]; L_2(K,V)))
\end{equation*}
where $1/r + 1/r' = 1$. We have used the fact that $[L^2_{\mathcal{F}}(\Omega\times[0,T],\lambda; H)]' = L^2_{\mathcal{F}}(\Omega\times [0,T],-\lambda; H')$ and that the dual of $L_2(K,V)$ can be identified with $L_2(K,V')$ (see e.g. \cite{aubin2011applied}, page 291).
\begin{remark} {\em If $B$ is a separable and reflexive Banach space, then the dual of $L^2_{\mathcal{F}}(\Omega\times[0,T]; B)$ is $L^2_{\mathcal{F}}(\Omega\times[0,T]; B')$ (cfr. e.g. \cite{joseph1977vector}). In our case, $E$ is the space of real continuous functions,
 so this result no longer holds. However it is still true that $L^2_{\mathcal{F}}(\Omega\times[0,T]; B') \subset L^2_{\mathcal{F}}(\Omega\times [0,T]; B)'$. For our purposes, this is sufficient because we need only to evaluate $\tau^*$ at $f$  and $\zeta$ which are much more regular.
}
\end{remark}
\begin{remark}{\em
A priori it could be possible to choose $\gamma$ with values in $E'$ instead of   $H$. However
this would introduce $\left[ L_{\mathcal{F}}^r(\Omega; L^2([0,T]; E))\right]'$, which is difficult to treat.
}
\end{remark}
We note that by definition the following equality holds
\begin{equation}\label{duality}
\E\int_0^T {}_{H'}\braket{p(t),\gamma(t)}_{H}dt + \E\int_0^T \braket{q(t),\eta(t)}_{L_2(K,V)}dt \\
= \E\int_0^T \langle f(t),y(t)\rangle_{E}dt + \E\langle \zeta,y(T)\rangle_{E},
\end{equation}
where $(p,q)=\tau^*(f,\zeta)$, and $(y(\cdot), y(T))$ are the solution process
and its terminal value of equation
\eqref{eq.y.mild} corresponding to forcing terms $(\gamma,\eta)$.

\begin{proof}[Proof of Theorem \ref{thm_1}]
 Since $u$ is an optimal control we have $J(u^{\varepsilon}) - J(u) \geq 0$. Thanks to the estimates given in Lemma \ref{l.estimates.spike}, we already obtained in Proposition \ref{prop.cost.expansion} that
\begin{equation}\label{intermedio}
o(\varepsilon) \leq \E\int_0^T \left[ \delta^{\varepsilon}L(t) + D_xL(t,X(t),u(t))Y^{\varepsilon}_t\right]dt + \E \left( D_xG(X_T)Y^{\varepsilon}_T(x) \right).
\end{equation}
Now we use duality (more precisely we use equation \eqref{duality} with $\eta = 0$) for the first variation equation \eqref{eq.first.variation}, which reads as
\begin{equation}
\E \braket{D_xG(X(T)),Y^{\varepsilon}(T)}_E =
 \E \int_0^T \left[ {}_{H'}\braket{p(t),\delta^{\varepsilon}F(t)}_H - \braket{D_xL(t,X(t),u(t)),Y^{\varepsilon}(t)}_E\right]dt.
\end{equation}
Here we have to check that $\delta^{\varepsilon}F \in L_{\mathcal{F}}^r(\Omega; L^2(([0,T],-\lambda; H))$ but this is true thanks to Hypothesis \ref{hp_f}.
Substituting in \eqref{intermedio} we get
\begin{equation}
o(1) \leq \dfrac{1}{\varepsilon}\E\int_{t_0}^{t_0 + \varepsilon}
\left[\mathcal{H}(t,X(t),u^\varepsilon(t),p(t)) - \mathcal{H}(t,X(t),u(t),p(t))\right]dt .
\end{equation}
Now the proof can be concluded by usual arguments, see for instance
\cite{yong1999stochastic}.
\\
\end{proof}

\section{The adjoint equation and the proof of Theorem \ref{thm_2}}
In this section we study a backward stochastic partial differential equation (BSPDE)
and we  characterize the adjoint processes $(p,q)$ as its unique solution.
This leads immediately to the proof of Theorem \ref{thm_2}.
While the process $p$ takes values in $E'$, it is not easy to formulate
the BSPDE as an equation in this space, since the usual tools of stochastic
calculus are not available there, in particular
there is no  version of the martingale representation theorem in $E'$.
We will therefore
embed $E'$ in a bigger Hilbert space in which we can use many standard
techniques and we can uniquely solve the BSPDE in a mild formulation.
Then we will prove that  the solution is indeed  more regular and in particular
that $p$ takes values in $E'$ (and even in $H'$) as desired.

Recall that if we identify $H \simeq H'$, we can embed the Gelfand triple $ E \subset H' \subset E'$
in a Hilbertian triple $ V \subset H' \subset V'$ and we obtain the following dense continuous inclusions
\[ V \subset E \subset H\simeq H' \subset E' \subset V',\]
where now we suppose that Hypothesis \ref{hp_V2} holds true.
We first write the dual BSPDE in a formal way as follows
\begin{equation}\label{formal}
\begin{sistema}
-dp(t) = \left[ A'p(t) + D_xF(X(t),u(t))'p(t)+ f(t))\right]dt - q(t)dW(t) \\
p(T) = \zeta,
\end{sistema}
\end{equation}
where $f(t):= D_xL(t,X(t),u(t))^*$ and $\zeta := D_xG(X(T))^*$
take values in $E'$.
This equation is given the following precise meaning in a mild formulation:
 for all $t \in [0,T]$, $\mP$-a.s.
\begin{equation}\label{eq.backward}
\begin{split}
p(t) = & S(T-t)'\zeta + \int_t^T S(s-t)'\left[ D_xF(X(s),u(s))^*p(s)+ f(s)\right]ds \\
&- \int_t^T S(s-t)^{\sim}q(s)dW(s),
\end{split}
\end{equation}
where  $S(t)':E'\to  H'$ is the adjoint of  $S(t):H\to  E$ (compare Hypothesis \ref{hp_eq}-1)
and
 $S(t)^{\sim}: V' \rightarrow V' $ is the adjoint of $S(t)$ viewed as an operator
 from $V$ to $V$, which is possible thanks to Hypothesis \ref{hp_V1}. On the other hand,
 $T^*:E'\to E'$ denotes the dual of any bounded linear operator $T:E\to E$.
Note that equality  \eqref{eq.backward} has a
 meaning in  the space $V'$, in which the stochastic convolution takes values.
\begin{remark}{\em
Notice that, with the previous notation, the adjoint of $S(t): H \rightarrow H$
(i.e., $S(t)$ viewed as an operator on $H$ rather than from $H$ to $E$)
coincides with
  $S(t)'|_{H'} :H'\to H'$, the restriction of $S(t)':E'\to H'$ to $H'\subset E'$.
}
\end{remark}
\begin{remark}{\em
Notice that the BSPDE is linear in $p$, but the map
$(t,\omega) \mapsto D_xF(X(t),u(t))^*: E' \rightarrow E'$ is not bounded.
Using the Nemytskii characterization it is easy to see that $D_xF(x,u)'$ acts on measures
 as multiplication by a density. Indeed, for any $v \in E$
\begin{equation}\label{eq.DF'}
\begin{split}
\braket{D_xF(x,u)^*p,v}_E &= \braket{p,D_xF(x,u)v}_E \\
& = \int_{\mathcal{O}} f'(x(\xi),u)v(\xi)dp(\xi)\\
&= \int_{\mathcal{O}} v(\xi)\cdot f'(x(\xi),u)dp(\xi)\\
\end{split}	
\end{equation}
which means that it maps $p$ into $ f'(x,u)p$.
It is worth noting that also the terminal condition lives in $E'$, as well as the forcing term $D_xL(t,X(t),u(t))^*$.
}
\end{remark}
Now we can state the following
\begin{thm}\label{t.ex_uniq_BSPDE}
Let Hypotheses \ref{hp_eq}, \ref{hp_V2}, \ref{hp_cost}, \ref{hp_f} hold
and $r'\in (1,2)$. Then there exists a unique mild solution $(p,q) \in L^{r'}_{\mathcal{F}}(\Omega; L^2(([0,T],\lambda;H'))\times L^{r'}_{\mathcal{F}}(\Omega; L^2(([0,T];\mathcal{L}_2(K,V')))$ to the BSPDE \eqref{eq.backward}.
\end{thm}

\begin{proof}
\textbf{Existence.}\\
STEP 1: (Regularization) We shall construct a solution to equation \eqref{eq.backward} by means of an approximation technique. The main tool we use is the smoothing effect of the semigroup $S(t)':E' \to H'$. If we define
\[ \zeta^{\varepsilon} := S(\varepsilon)'\zeta, \qquad f^{\varepsilon}(t) := S(\varepsilon)'f(t),\]
then both $\zeta^{\varepsilon}, f^{\varepsilon}(t)$ take values in $H'$, for all $\varepsilon>0$.
We also introduce the Yosida approximations $f_{\alpha}$ of $f$. Thanks to Lemma \ref{l.yosida} we know that the associated Nemytskii operator $F_{\alpha}(x,u)$ is Lipschitz in $x$ with
respect to the norm of $H$, hence $\norm{D_xF_{\alpha}}_{\mathcal{L}(H)}$
is uniformly bounded by some costant $k_{\alpha}$ which depends only on $\alpha$.
If we set $T_{\alpha}(t): = D_xF_{\alpha}(X(t),u(t))$ then its adjoint
$T_{\alpha}(t)^*$ is a bounded process with values in $\mathcal{L}(H')$.
Next we introduce an approximate BSPDE that we first write in a formal way as
\begin{equation}\label{BSDE.approx}
\begin{sistema}
-dp^{\varepsilon,\alpha}(t) = \left[ A'p^{\varepsilon,\alpha}(t) + T_{\alpha}(t)^*p^{\varepsilon,\alpha}(t)+ f^{\varepsilon}(t))\right]dt - q^{\varepsilon,\alpha}(t)dW(t) , \\
p^{\varepsilon,\alpha}(T) = \zeta^{\varepsilon}.
\end{sistema}
\end{equation}
By the result in \cite{hu1991adapted} there exists a unique mild solution to this equation, i.e. a  process
\[(p^{\varepsilon,\alpha},q^{\varepsilon,\alpha}) \in L_{\mathcal{F}}^2(\Omega\times
[0,T],H') \times L_{\mathcal{F}}^2([0,T],\mathcal{L}_2(K,H')),\]
such that for all $t>0$, $\mP$-a.s
\begin{equation}
 p^{\varepsilon,\alpha}(t)  = S(T-t)'\zeta^{\varepsilon} + \int_t^T S(s-t)'\big[T_{\alpha}(s)^*p^{\varepsilon,\alpha}(s) +f^{\varepsilon}(s)\big]ds  +\int_t^T S(s-t)' q^{\varepsilon,\alpha}(s)dW(s).
\end{equation}
Our aim now is to prove a uniform estimate for the approximate solution via a duality argument.\\
STEP 2: (Duality) For all $\gamma \in L_{\mathcal{F}}^r(\Omega; C([0,T],H))$ and $\eta \in L_{\mathcal{F}}^r(\Omega; L^2([0,T],\mathcal{L}_2(K,V)))$ let us consider the equation
\[ y^{\alpha}(t)  = \int_0^t S(t-s)\big[T_{\alpha}(s)y^{\alpha}(s) +\gamma(s)\big]ds  +\int_0^t S(t-s)\eta(s)dW(s), \qquad t \in [0,T], \]
which has a unique mild solution $y^{\alpha}(t)$ in the sense of
 Proposition \ref{prop.forward.duale}.  For all $n \in \mathbb{N}$,
 let $A_n = nA(n-A)^{-1}$ denote
  the Yosida approximations of the operator $A$ and $S_n(t) := e^{tA_n}$. Then, in a similar way,
  we can find a solution to
\[ y_n^{\alpha}(t)  = \int_0^t S_n(t-s)\big[T_{\alpha}(s)y_n^{\alpha}(s) +\gamma(s)\big]ds  +\int_0^t S_n(t-s)\eta(s)dW(s). \]
Again by   \cite{hu1991adapted} there exists a unique solution to
  the  equation
\[ p^{\varepsilon,\alpha}_n(t)  = S_n(T-t)'\zeta^{\varepsilon} + \int_t^T S_n(T-t)'\big[T_{\alpha}(s)^*p^{\varepsilon,\alpha}_n(s) +f^{\varepsilon}(s)\big]ds  +\int_t^T S_n(T-t)'q^{\varepsilon,\alpha}(s)dW(s), \]
in the space $L_{\mathcal{F}}^2(\Omega\times
[0,T],H') \times L_{\mathcal{F}}^2([0,T],\mathcal{L}_2(K,H'))$. Notice that,
since $A_n$ are bounded operators, the equations
 can now be written in the stronger form
\begin{equation*}
\begin{sistema}
dy^{\alpha}_n(t) = \big[ A_n y^{\alpha}_n(t) + T_{\alpha}(t)y_n^{\alpha}(t) +\gamma(t)\big]dt +  \eta(t)dW(t) ,\\
y^{\alpha}_n(0) = 0;
\end{sistema}
\end{equation*}
\begin{equation*}
\begin{sistema}
- dp^{\varepsilon, \alpha}_n(t) = \big[ A_n^* p^{\varepsilon, \alpha}_n(t) + T_{\alpha}(t)^*p^{\varepsilon,\alpha}_n(t) +f^{\varepsilon}(t)\big]dt - q^{\varepsilon, \alpha}(t)dW(t) , \\
p^{\varepsilon, \alpha}_n(T) = \zeta^{\varepsilon},
\end{sistema}
\end{equation*}
where $A_n^*:H'\to H'$ is understood as the adjoint of $A_n:H\to H$.
Computing the Ito
differential
$d(_{H}\braket{y_n^{\alpha}(t),p_n^{\varepsilon,\alpha}(t)}_{H'})$ and letting $n \rightarrow \infty$ we get by standard arguments (see e.g. \citep{tessitore1996existence}):
\begin{equation}
\begin{split}
&\E\int_0^T {}_{H'}\braket{p^{\varepsilon,\alpha}(t),\gamma(t)}_{H}dt + \E\int_0^T \braket{q^{\varepsilon,\alpha}(t),\eta(t)}_{\mathcal{L}_2(K,H)}dt \\
&= \E\int_0^T {}_{H'}\braket{f^{\varepsilon}(t),y^{\alpha}(t)}_{H}dt + \E {}_{H'}\braket{\zeta^{\varepsilon},y^{\alpha}(T)}_{H}.
\end{split}
\end{equation}
Since, by  Proposition \ref{prop.forward.duale},
 $y^\alpha{}(t)$ takes values in $E$ we also have
\begin{equation}\label{eq.duality.approx}
\begin{split}
&\E\int_0^T {}_{H'}\braket{p^{\varepsilon,\alpha}(t),\gamma(t)}_{H}dt + \E\int_0^T \braket{q^{\varepsilon,\alpha}(t),\eta(t)}_{\mathcal{L}_2(K,H)}dt \\
&= \E\int_0^T {}_{E'}\langle f^{\varepsilon}(t),y^{\alpha}(t)\rangle_{E}dt + \E {}_{E'}\langle \zeta^{\varepsilon},y^{\alpha}(T)\rangle_{E}.
\end{split}
\end{equation}
Let us now define the set $\mathcal{A} := \lbrace \gamma \in L^r_{\mathcal{F}}(\Omega; L^2([0,T],-\lambda, H)): \norm{\gamma}_{L^r_{\mathcal{F}}(\Omega; L^2([0,T],-\lambda, H))} \leq 1 \rbrace$ and $\mathcal{\tilde{A}} := \lbrace \gamma \in L^r_{\mathcal{F}}(\Omega; L^2([0,T],-\lambda, H)) \cap L_{\mathcal{F}}^r(\Omega; C([0,T],H)): \norm{\gamma}_{L^r_{\mathcal{F}}(\Omega; L^2([0,T],-\lambda, H))} \leq 1, \rbrace$.
It easy to see that $\mathcal{\tilde{A}}$ is densely embedded in $\mathcal{A}$.
If we take $\eta = 0$ in \eqref{eq.duality.approx} we get
\begin{equation*}
\begin{split}
&\left( \E\left( \int_0^T\abs{p^{\varepsilon,\alpha}(t)}^2_{H'}(T-t)^{\lambda}dt \right)^{r'/2} \right)^{1/r'} \\
&\leq \sup_{\gamma \in \mathcal{A}} \left[ \E\int_0^T \langle f^{\varepsilon}(t),y^\alpha(t)\rangle_{E}dt + \E \langle \zeta^{\varepsilon},y^\alpha(T)\rangle_{E} \right]\\
&= \sup_{\gamma \in \mathcal{\tilde{A}}} \left[ \E\int_0^T \langle f^{\varepsilon}(t),y^\alpha(t)\rangle_{E}dt + \E \langle \zeta^{\varepsilon},y^\alpha(T)\rangle_{E} \right]\\
&\leq \sup_{\gamma \in \mathcal{\tilde{A}}} \left[ \left(\E\int_0^T \abs{f^{\varepsilon}(t)}_{E'}^2 dt\right)^{1/2}\left(\E\int_0^T \abs{y^\alpha(t)}_{E}^2 dt\right)^{1/2}\right]\\
&+ \sup_{\gamma \in \mathcal{\tilde{A}}} \left[ \left(\E\abs{\zeta^{\varepsilon}}_{E'}^2\right)^{1/2} \left(\E\abs{y^\alpha(T)}_{E}^2\right)^{1/2}\right]\\
&\leq C\left[ \left(\E\int_0^T \abs{f^{\varepsilon}(t)}_{E'}^2 dt\right)^{1/2} + \left(\E\abs{\zeta^{\varepsilon}}_{E'}^2\right)^{1/2} \right]
\end{split}
\end{equation*}
where we used H\"{o}lder's inequality and the estimates \eqref{estimate.forward} and \eqref{estimate.forward.T} with $t = T$ of Proposition \ref{prop.forward.duale} for the process $y^{\alpha}(t)$. Then, noting that $\abs{f^{\varepsilon}}_{E'} \leq \abs{f}_{E'}$ and $\abs{\zeta^{\varepsilon}}_{E'} \leq \abs{\zeta}_{E'}$, we obtain
\begin{equation*}
\left( \E \left(\int_0^T\abs{p^{\varepsilon,\alpha}(t)}^2_{H'}(T-t)^{\lambda}dt\right)^{r'/2}\right)^{1/r'} \leq C\left[ \E\int_0^T \abs{f(t)}_{E'}^2 dt  + \E\abs{\zeta}_{E'}^2\right]^{1/2}
\end{equation*}
where $C$ does not depend on $\varepsilon$ and $\alpha$. Let us now
look for a similar estimate on $q^{\varepsilon,\alpha}(t)$. We first
recall   that if $(e_k)_{k \in \mathbb{N}}$ is a complete orthonormal system of $H$ such that $e_k \in V$ for all $k \in \mathbb{N}$, we have
\begin{equation}
\begin{split}
\braket{q^{\varepsilon,\alpha}(t),\eta(t)}_{\mathcal{L}_2(K,H)} &= \sum_{k} \braket{q^{\varepsilon,\alpha}(t)e_k, \eta(t)e_k}_{H} \\
&= \sum_{k} {}_{V'}\langle q^{\varepsilon,\alpha}(t)e_k, \eta(t)e_k\rangle_{V}\\
&= {}_{\mathcal{L}_2(K,V')}\langle q^{\varepsilon,\alpha}(t), \eta(t)\rangle_{\mathcal{L}_2(K,V)}.
\end{split}
\end{equation}
Hence we can identify $\mathcal{L}_2(K,V)'$ with the space $\mathcal{L}_2(K,V')$: for a detailed proof see \citep{aubin2011applied}, page 291. Writing now \eqref{eq.duality.approx}
 with $\gamma = 0$,
letting $\mathcal{B}: = \lbrace \eta \in L^{r'}_{\mathcal{F}}(\Omega;L^2([0,T],\mathcal{L}_2(K,V))) : \norm{\eta}_{L^{r'}_{\mathcal{F}}(\Omega;L^2([0,T],\mathcal{L}_2(K,V)))} \le 1\rbrace$
 and recalling the estimate \eqref{estimate.forward}  and \eqref{estimate.forward.T} with $t = T$ we get
\begin{equation*}
\begin{split}
&\left( \E\left( \int_0^T\abs{q^{\varepsilon,\alpha}(t)}^2_{\mathcal{L}_2(K,V')}dt \right)^{r'/2} \right)^{1/r'} \\
&\leq \sup_{\eta \in \mathcal{B}} \left[ \E\int_0^T  {}_{E'} \langle f^{\varepsilon}(t),y^{\alpha}(t)\rangle_{E}dt + \E {}_{E'}\langle \zeta^{\varepsilon},y^\alpha(T)\rangle_{E} \right]\\
&\leq \sup_{\eta \in \mathcal{B}}  \left[ \left(\E\int_0^T \abs{f^{\varepsilon}(t)}_{E'}^2 dt\right)^{1/2}\left(\E\int_0^T \abs{y^\alpha(t)}_{E}^2 dt\right)^{1/2} \right]\\
&+ \sup_{\eta \in \mathcal{B}}  \left[  \left(\E\abs{\zeta^{\varepsilon}}_{E'}^2\right)^{1/2} \left(\E\abs{y^\alpha(T)}_{E}^2\right)^{1/2} \right]\\
&\leq C\left[ \left(\E\int_0^T \abs{f^{\varepsilon}(t)}_{E'}^2 dt\right)^{1/2} + \left(\E\abs{\zeta^{\varepsilon}}_{E'}^2\right)^{1/2} \right]
\end{split}
\end{equation*}
Summarizing,
 we have obtained the following  estimate
\begin{equation}\label{eq.approx.duality}
\norm{p^{\varepsilon,\alpha}}_{L^{r'}_{\mathcal{F}}(\Omega;L^2([0,T],\lambda,H'))}+ \norm{q^{\varepsilon,\alpha}}_{L^{r'}_{\mathcal{F}}(\Omega;L^2([0,T],\mathcal{L}_2(K,V')))}   \leq C\left[ \E\int_0^T \abs{f(t)}_{E'}^2 dt  + \E\abs{\xi}_{E'}^2\right].
\end{equation}
STEP 3: (Convergence) By \eqref{eq.approx.duality}  there exists a sequence
$(\varepsilon_k,\alpha_k) \to (0,0)$, as $k \to \infty$, such that
$ (p^{\varepsilon_k,\alpha_k},q^{\varepsilon_k,\alpha_k})=:(p_k,q_k) $ converges
to a limit  $(p,q)$
 weakly in the product space. Setting for brevity
$T(t)= D_xF(X(t),u(t))$,
we will show that $(p,q)$ is a mild solution to equation
\eqref{eq.backward}, which we now write in the form
\begin{equation}
p(t) = S(T-t)'\zeta + \int_t^T S(s-t)'\left[ T(t)^*p(s)+ f(s)\right]ds - \int_t^T S(s-t)'q(s)dW(s).
\end{equation}
First we have  $S(T-t)'\zeta^{\varepsilon_k} \rightarrow S(T-t)'\zeta$ weakly in $H'$ since, for any $h \in H$,
$$
{}_{H'}\braket{S(T-t)'\zeta^{\varepsilon_k},h}_H = {}_{E'}\braket{\zeta, S(\varepsilon_k)S(T-t)h}_E
\to {}_{E'}\braket{\zeta,S(T-t)h}_E
 = {}_{H'}\braket{S(T-t)'\zeta,h}_H.
$$
Then we address
 the drift term. We take $h \in H$ and we have
\begin{equation*}
\begin{split}
\E\int_t^T {}_{H'}\braket{S(s-t)'T_{\alpha_k}(s)^*p_k(s),h}_H ds &= \E\int_t^T {}_{H'}\braket{T_{\alpha_k}(s)^*p_k(s),S(s-t)h}_H ds \\
& = \E\int_t^T {}_{H'}\braket{p_k(s),\left[T_{\alpha_k}(s) - T(s)\right]S(s-t)h}_H ds \\
&+ \E\int_t^T {}_{H'}\braket{p_k(s),T(s)S(s-t)h}_H ds.\\
\end{split}
\end{equation*}
Let us consider the first term on the right hand side:
\begin{equation}\label{convnemi}
\begin{split}
\E\int_t^T &{}_{H'}\braket{p_k(s),\left[T_{\alpha_k}(s) - T(s)\right]S(s-t)h}_H ds \\
&= \E\int_t^T \braket{(T-s)^{\lambda/2}p_k(s),(T-s)^{-\lambda/2}\left[T_{\alpha_k}(s) - T(s)\right]S(s-t)h}_H ds \\
&\leq \left( \E \left( \int_t^T (T-s)^{\lambda}\abs{p_k(s)}^2 ds\right)^{r'/2} \right)^{1/r'} \cdot \\
&\cdot \left( \E \left( \int_t^T (T-s)^{-\lambda}\abs{\left[T_{\alpha_k}(s) - T(s)\right]S(s-t)h}_H^2 ds \right)^{r/2} \right)^{1/r} \\
&\leq K \cdot \left( \E \left( \int_t^T (T-s)^{-\lambda}\abs{\left[T_{\alpha_k}(s) - T(s)\right]S(s-t)h}_H^2 ds \right)^{r/2} \right)^{1/r}.
\end{split}
\end{equation}
Since $S(s-t)h \in E$ (for $s\neq t$) then $\abs{T_{\alpha_k}(s)S(s-t)h - T(s)S(s-t)h}_E \to 0$
for all $\omega \in \Omega$, $s \in [0,T]$. Next we have
\begin{equation}
\begin{split}
\abs{T(s)S(s-t)h}^2_H &= \int_{\mathcal{O}}\abs{f'(X(s,\xi),u(s))\left[S(s-t)h\right](\xi)}^2m(dx) \\
&\leq \left(\sup_{\xi \in \mathcal{O}}\abs{f'(X(s,\xi),u(s))}\right)^2 \abs{S(s-t)h}^2_H \\
&\leq C_F(\omega)\cdot \abs{h}^2.
\end{split}
\end{equation}
where we have used the polynomial growth condition
\eqref{polygrowth} on $f'$ and
$C_F(\omega)$ was introduced in \eqref{crescita.DF}.
Using \eqref{polygrowthdue} instead of \eqref{polygrowth}
we prove by similar passages that $\abs{T_\alpha(s)S(s-t)h}^2_H \leq C_F(\omega)\cdot \abs{h}^2$.
Thanks to the assumption that $\lambda < 1$ the right-hand side of \eqref{convnemi} then tends to zero by dominated
convergence.

To prove that $\E\int_t^T {}_{H'}\braket{p_k(s),T(s)S(s-t)h}_H ds
\to \E\int_t^T {}_{H'}\braket{p(s),T(s)S(s-t)h}_H ds$ it is enough to notice
that $s\to T(s)S(s-t)h\,1_{s\ge t}$ lies in the space
$L^{r}_{\mathcal{F}}(\Omega;L^2([0,T],-\lambda,H))$ and to use
the weak convergence property of $p_k$.

Finally,
regarding the stochastic convolution it is enough to notice that the map
\[ q \longmapsto \int_t^T S(s-t)^\sim q(s)dW(s)\]
is linear and continuous from $L^{r'}_{\mathcal{F}}(\Omega;  L^2([0,T]; \mathcal{L}_2(K,V')))$ to $L^{r'}(\Omega,V')$, hence weakly continuous. Indeed, for fixed $t$, let us define:
\begin{equation*}
I(\bar{q}) := \int_0^T \left( S(s-t)^\sim 1_{\lbrace s > t \rbrace} \right) \bar{q}(s) dW(s)
\end{equation*}
Then, using Burkholder-Davis-Gundy inequality we get
\begin{equation*}
\begin{split}
\left(\E \abs{I(\bar{q})}^{r'}\right)^{1/r'} &\leq K \left( \E \left( \int_0^T \abs{S(s-t)^\sim 1_{\lbrace s > t \rbrace} \bar{q}(s)}^2_{\mathcal{L}_2(K,V')}ds \right)^{r'/2} \right)^{1/r'} \\
&\leq K \cdot \norm{\bar{q}}_{L^{r'}(\Omega; L^2([0,T];\mathcal{L}_2(K,V')))} < \infty,
\end{split}
\end{equation*}
thanks to Hypothesis \ref{hp_V2}. Then the existence part follows.\\
\newline
\textbf{Uniqueness.}\\
The BSPDE \eqref{eq.backward} is linear, then it is enough to prove that if the forcing term  $f(t)$ and the final condition $\zeta(t)$  are zero then also the solution is identically zero.
Hence we are dealing with the following
\begin{equation}\label{BSDEperunicita}
p(t) = \int_t^T S(s-t)'T(s)^*p(s)ds - \int_t^T S(s-t)'q(s)dW(s).
\end{equation}
We apply the operator    $S(1/n)'$ to both sides   and   setting $p^n:=S(1/n)'p$,  $q^n:=S(1/n)'q$ we
obtain
\begin{equation}\label{approxuniq}
p^n(t) = \int_t^T S(s-t)'S(1/n)'T(s)^*p(s)ds - \int_t^T S(s-t)'q^n(s)dW(s).
\end{equation}
Notice that
\begin{equation*}
\begin{split}
\abs{S(1/n)' T(s)^*p(s)}_{H'} &\leq \abs{S(1/n)'}_{\mathcal{L}(E',H')}\abs{T(s)^*p(s)}_{E'}\\
&\leq \abs{S(1/n)'}_{\mathcal{L}(E',H')} \abs{T(s)}_{\mathcal{L}(E)}\abs{p(s)}_{E'} \\
&\leq C_n(1+\abs{X_s}_E^k)\abs{p(s)}_{H'} \\
&\leq C_n(\omega)\abs{p(s)}_{H'}, \quad \mP-\text{a.s.}
\end{split}
\end{equation*}
Moreover we have the following integrability condition:  for $\varepsilon >9$,
\begin{equation}\label{eq.int_drift}
\int_0^{T-\varepsilon}\abs{p(s)}_{H'} dt < \infty, \quad \mP-\text{a.s.}
\end{equation}
thanks to the fact that $p \in L_{\mathcal{F}}^{r'}(\Omega; L^2([0,T],\lambda;H'))$. Notice that 
\eqref{eq.int_drift} fails in general for $\varepsilon=0$.
Next, by the semigroup property, 
$S(1/n)=S(1/2n)S(1/2n)$.
By Hypothesis \ref{hp_eq}-1 we have  $S(1/2n)(H)\subset  E$. Now we exploit
  Hypothesis \ref{hp_V2} to deduce that
$S(1/2n) (E) \subset  V$ and we conclude that  
$S(1/n)(H)\subset V$ and even that $S(1/n)\in\mathcal{L}(H,V)$,
by the closed graph theorem. It follows that
$q^n(s)\in\mathcal{L}_2(K,H') $ and we have the
estimate 
\begin{equation}\label{eq.int_diffusion}
\abs{q^n(s)}_{\mathcal{L}_2(K,H')} \leq \abs{S(1/n)'}_{\mathcal{L}(V',H')}\abs{q(s)}_{\mathcal{L}_2(K,V')} = C_n\abs{q(s)}_{\mathcal{L}_2(K,V')},
\end{equation}
 which guarantees that the stochastic convolution in 
 \eqref{approxuniq}
  is well defined.
We write equation \eqref{approxuniq} on the time interval $[0,T-\epsilon]$
obtaining
\begin{equation*}
p^n(t) = S(T-\varepsilon - t)'\zeta^n_{\varepsilon} + \int_t^{T-\varepsilon} S(s-t)'S(1/n)' T(s)^*p(s)ds - \int_t^{T -\varepsilon}S(s-t)'q^n(s)dW(s),
\end{equation*}
where $\zeta^n_\varepsilon: = p^n(T-\varepsilon)$ is the value of the solution at time $T-\varepsilon$.
Next we fix an arbitrary $\gamma$  in the space $ L_{\mathcal{F}}^2(\Omega\times[0,T],E) $
and define
\begin{equation*}
y(t) = \int_0^t S(t-s)\left[ T(s)y(s) + \gamma(s)\right]ds.
\end{equation*}
Note that by Proposition \ref{prop.forward.duale} $y$ is the solution to
the equation \eqref{eq.y.mild} with $\eta=0$.

Replacing the operator $A$ with the Yosida approximations $A_m$ we call
$p^{m,n}$, $ y^{m}$ the solutions to the corresponding equations.
We are in a position to
 apply the It\^{o} formula to $\braket{p^{m,n},y^{m}}$, thanks to the integrability conditions obtained in \eqref{eq.int_drift} and \eqref{eq.int_diffusion}.
Passing to the limit as $m \to \infty$ and taking the expectation we obtain, see \citep{tessitore1996existence} for the details,
\begin{equation}\label{approxneps}
\begin{split}
&\E\int_0^{T-\varepsilon} {}_{H'}\braket{p^n(t),\gamma(t)}_H dt + \E\int_0^{T-\varepsilon}\left[ {}_{H'}\braket{p^n,T(t)y(t)}_H - {}_{H'}\braket{S(1/n)' T(t)^*p(t),y(t)}_H \right]dt \\
&= \E{}_{H'}\braket{y(T-\varepsilon),\zeta^n_{\varepsilon}}_H.
\end{split}
\end{equation}
Now we pass to the limit letting $n\rightarrow \infty$. We have
\begin{equation}
\begin{split}
&\E\int_0^{T-\varepsilon} {}_{H'}\braket{p^n(t),\gamma(t)}_H dt 
\to \E\int_0^{T-\varepsilon} {}_{H'}\braket{p(t),\gamma(t)}_H dt ,
\\
&  \E\int_0^{T-\varepsilon} {}_{H'}\braket{p^n(t),T(t)y(t)}_H dt \to
    \E\int_0^{T-\varepsilon} {}_{H'}\braket{p(t),T(t)y(t)}_H dt ,
\end{split}
\end{equation}
by dominated convergence, due to the 
fact that $p(t)$ takes values in $H'$.
We also have
\begin{equation*}
 \E\int_0^{T-\varepsilon}  {}_{E'}\braket{T(t)^*p(t),y(t)}_E \to
  \E\int_0^{T-\varepsilon}{}_{E'}\braket{T(t)^*p(t),S(1/n)y(t)}_E\,dt
\end{equation*}
Indeed, from the analyticity of the semigroup in $E$, see \cite{lunardi2012analytic},
it follows that almost all paths of the process
 $y$ take values in $\overline{D(A)}$ (the closure of the domain of $A$
 in $E$), so that
 $S(1/n)y(t)\to y(t)$ in $E$ 
and the desired conclusion follows from the dominated convergence theorem. 
So from \eqref{approxneps},  letting $n\rightarrow \infty$, it follows that
\begin{equation}\label{eq_duality_p_tilde}
\E\int_0^{T-\varepsilon} \braket{p(t),\gamma(t)}dt = \E\braket{y(T-\varepsilon),p(T-\varepsilon)}.
\end{equation}
Setting $t=T-\epsilon$ in \eqref{BSDEperunicita} and taking the conditional
expectation given $\mathcal{F}_{T-\varepsilon}$ we obtain
$p(T-\varepsilon)=\E \,[\int_{T-\varepsilon}^T T(s)^*p(s)ds
\mid {\mathcal{F}_{T-\varepsilon}}]$
so that
\begin{equation}
\begin{split}
\E\int_0^{T-\varepsilon} \braket{p(t),\gamma(t)}dt &
= \E\braket{y(T-\varepsilon),
\E\,[\int_{T-\varepsilon}^T T(s)^*p(s)ds\mid {\mathcal{F}_{T-\varepsilon}}]
} \\
&= \E \int_{T-\varepsilon}^T \braket{y(T-\varepsilon), T(s)^*p(s)}ds \\
\end{split}
\end{equation}
which converges to $0$ if $\varepsilon \rightarrow 0$.  We finally obtain
$\E\int_0^T \braket{p(t),\gamma(t)}dt = 0$ 
and we get the result by the arbitrariness of $\gamma$.
\end{proof}
We immediately arrive at the following conclusion.
\begin{proof}[Proof of Theorem \ref{thm_2}]
Thanks to Theorem \ref{t.ex_uniq_BSPDE}, there exist a unique process $p(\cdot)$ which is the first component of the solution to the BSPDE \eqref{eq.backward}. The conclusion follows using Theorem \ref{thm_1}.
\end{proof}

\bibliography{mybib}
\bibliographystyle{plain}

\end{document}